\documentclass[12pt,a4paper]{article}
\usepackage{graphicx}
\usepackage[ngerman,english]{babel}
\usepackage[utf8x]{inputenc}
\usepackage{verbatim} 
\usepackage{epstopdf}
\usepackage{authblk} 

\usepackage{todonotes} 
\usepackage{eso-pic} 
\usepackage{verbatim} 
\usepackage[numbers]{natbib} 

\usepackage{ucs}
\usepackage[T1]{fontenc} 
\usepackage{lmodern}
\usepackage{amsmath,amsthm,amssymb}
\usepackage{calligra}
\usepackage{wrapfig}
\usepackage{caption} 
\usepackage{enumitem} 
\usepackage{multirow}
\usepackage{tikz}
\usetikzlibrary{calc}
\usetikzlibrary{decorations.pathreplacing} 
\tikzstyle{vertex}=[circle,fill=black!25,minimum size=15pt,inner sep=0pt]
\tikzstyle{vertexSmallGray}=[circle,gray!75,fill=gray!75,minimum size=4pt,inner 
sep=0pt, draw] 
\tikzstyle{vertexSmallWhite}=[circle,fill=white,minimum size=4pt,inner sep=0pt, draw]  
\tikzstyle{vertexSmallBlack}=[circle,black,fill=black,minimum 
size=4pt,inner 
sep=0pt, 
draw]
\tikzstyle{vertexBigBlack}=[circle,fill=black,minimum size=10pt,inner sep=0pt, 
draw]

\usepackage{aliascnt} 
\usepackage{hyperref} 
\addto\extrasenglish{}
\addto\extrasenglish{}

\newtheorem{theorem}{Theorem}[section]

\newaliascnt{lemma}{theorem}
\newtheorem{lemma}[lemma]{Lemma}
\aliascntresetthe{lemma}

\newaliascnt{conjecture}{theorem}
\newtheorem{conjecture}[conjecture]{Conjecture}
\aliascntresetthe{conjecture}

\theoremstyle{definition}
\newaliascnt{definition}{theorem}
\newtheorem{definition}[definition]{Definition}
\aliascntresetthe{definition}

\newcommand{\nz}{\mathbb{N}}

\providecommand{\keywords}[1]{\textbf{Keywords:} #1}

\begin{document}


\title{On $k$-Bend and Monotonic $\ell$-Bend Edge Intersection 
Graphs
of Paths on a Grid}
\author[1]{Eranda \c{C}ela}
\author[2]{Elisabeth Gaar\thanks{The second author acknowledges support by the 
Austrian Science Fund (FWF): I 3199-N31. 
We are very thankful to the anonymous 
reviewers for their careful reading and their valuable suggestions.}}
\affil[1]{TU Graz, Steyrergasse 30, Graz A-8010, Austria, 
\href{mailto:cela@math.tugraz.at}{cela@math.tugraz.at}}
\affil[2]{Johannes Kepler University Linz, Altenberger Strasse 69, 
Linz A-4040, Austria,  
\href{mailto:elisabeth.gaar@jku.at}{elisabeth.gaar@jku.at}}
\date{}

\maketitle

\begin{abstract}

If a graph $G$ can be represented by means of  paths on a grid, 
such that each vertex of $G$ corresponds to one path on the grid 
and two vertices of $G$ are adjacent if and only if the corresponding paths 
share a grid edge,
then this graph is called EPG and the representation is called EPG 
representation.
A $k$-bend EPG representation is an EPG representation in which each path has 
at most $k$ bends.
The class of all graphs that have a $k$-bend EPG representation is denoted by 
$B_k$. 
$B_\ell^m$ is the class of all graphs that have a monotonic  $\ell$-bend EPG
representation, i.e.\ an  $\ell$-bend EPG
representation,  where each path is ascending in both columns and rows. 

It is trivial that $B^m_k\subseteq B_k$ for all $k$. Moreover, it is
known that $B^m_k\subsetneqq B_k$, for $k=1$. By investigating the
$B_k$-membership and the $B^m_k$-membership of complete bipartite graphs we prove
that the inclusion is also proper for $k\in \{2,3,5\}$ and  for $k\geqslant 7$.
 In particular, we derive necessary conditions for
this membership that have to be fulfilled by $m$, $n$ and $k$, where $m$ and 
$n$ are the number of vertices on the two partition classes of the bipartite 
graph.
We conjecture that $B_{k}^{m} \subsetneqq 
B_{k}$ holds also for $k\in \{4,6\}$.

Furthermore, we show that $B_k  \not\subseteq B_{2k-9}^m$ holds for all 
$k\geqslant 5$. This implies that
 restricting the shape of the paths can lead to a significant increase of 
the number of bends needed in an EPG representation. 
So far no  bounds on the amount of that increase were known.
We prove that $B_1 \subseteq B_3^m$
holds, providing the first result of this kind.

\end{abstract}
\keywords{paths on a grid, EPG graph, (monotonic) bend number, complete 
bipartite graph}


\section{Introduction and Definitions}\label{intro:sec}

In 2009  Golumbic, Lipshteyn and
Stern~\cite{startpaper} introduced edge intersection graphs of paths on a grid.
If a graph $G$ can be represented by means of  paths on a grid, 
such that each vertex of $G$ corresponds to one path on the grid 
and two vertices of $G$ are adjacent if and only if the corresponding paths share a grid edge,
then this graph is called \emph{edge intersection graph of paths on a grid} \emph{(EPG)} 
and the representation is called \emph{EPG representation}. 
Here the term \emph{edge intersection of paths} refers to the fact that the paths share a grid edge.

A \emph{$k$-bend EPG representation} or \emph{$B_{k}$-EPG representation} is an 
EPG representation 
in which each path has at most $k$ bends. 
A graph that has a $B_{k}$-EPG representation is called \emph{$B_k$-EPG} and 
the class of all $B_k$-EPG graphs is denoted by \emph{$B_k$}.
We consider the following natural ordering of grid lines: the columns  
increase from the left to
the right and the rows increase from the  bottom to the  top. 
A path on a grid is called \emph{monotonic}, if it is ascending in both columns
and rows, i.e.\ it has the shape of a staircase that is going upwards from the
left to the right.
The graphs that have a $B_{\ell}$-EPG representation in which each path is 
monotonic are called \emph{$B_\ell^m$-EPG} and the class of all these graphs is 
denoted by \emph{$B_\ell^m$}.
The bend number $b(G)$ of a graph $G$ is the minimum $k$ such that $G$ is $B_k$-EPG. 
The monotonic bend number $b^m(G)$ of graph $G$   is defined as the minimum $\ell$ such that $G$ is 
$B_\ell^m$-EPG.
Note that already Golumbic, Lipshteyn and Stern~\cite{startpaper} showed that each graph is $B_{k}$-EPG and $B_{\ell}^{m}$-EPG for some $k$ and $\ell$.

As described in~\cite{startpaper}  the motivation for investigating EPG graphs
 was initially related to  applications from circuit
layout setting and  chip manufacturing. 
 In the knock-knee circuit layout model the wires can be seen as paths on a grid which
can cross and bend at a grid point but are not allowed to share a grid edge,
see \cite{knockknee, AsurveyOnWiring}. 
The wires can be put  in multiple layers each of them being a grid and 
such that the wires of each layer do not share a grid edge. 
In this setting the minimum number of layers needed to accomodate all wires
would be equal to the   chromatic number of the corresponding graph. 
Consider now that  a so-called transition hole is needed, whenever a wire bends. 
 If   a large number of  transition holes is included,  the layout area and
 consequently,  the cost of the chip,  may increase.
Therefore, it might be desirable to find a circuit layout setting which
 minimizes the largest number of bends used in each wire.  
In our notation this corresponds to finding  the minimum $k$ such that the corresponding graph is in $B_{k}$.

Similar graph classes  known in the literature include  
 \emph{edge intersection graphs of
paths on a tree} (EPT) (see \cite{GoJa85EPT}), \emph{vertex intersection graph of paths on a tree
(VPT)} (see~\cite{GoJa85EPTVPT}) and \emph{vertex
    intersection graphs of paths on a grid} (\emph{VPG}) (see~\cite{VPG}).
In this paper we will only deal with EPG graphs.

There has been a lot of research on EPG graphs since their introduction.
One of the topics of interest is the recognition problem of $B_k$-EPG graphs, 
i.e.\ to determine for a given $k$ and a given graph whether this graph is in 
$B_k$ ($B_k^m$).
Currently it is known that the recognition problem is NP-hard 
for $B_1$ (Heldt, Knauer and Ueckerdt~\cite{F}), 
$B_1^m$ (Cameron, Chaplick and Ho\`{a}ng~\cite{E}), 
$B_2$ and $B_2^m$ (Pergel and 
Rz{\k{a}}{\.{z}}ewski~\cite{LRecognizionB2EPGisNPComplete}).

Recently a number of results  on  combinatorial optimization problems on 
specific
$B_k$-EPG graphs have been published. Subject of investigation are  certain NP-hard  
combinatorial optimization problems which turn out to be  
tractable, i.e.\ polynomially solvable or approximable within  a guaranteed
approximation ratio,  for
$B_k$-EPG graphs, 
see~\cite{OCliqueColoringB1,NMaxIndependentSetInB1EPG,RMaxCliqueInB2,J}. 
Thus, the computation of the bend number and the monotonic bend
number of graphs or related  upper bounds is a relevant research question in
this context.  
However,  this appears to be a challenging task,
considering that even  the recognition  of $B_k$ ($B_k^m$)   graphs is
 NP-hard for $k=1$ and $k=2$, as mentioned above. 
 
 A related and more viable
 line of research is the determination of (upper bounds on)  the (monotonic) 
 bend number of
 special  graph classes.
Among the first graph class for which an upper bound on the bend number was given were planar graphs. 
The first upper bound was  $5$ and it was obtained in 2009 by  Biedl and
 Stern~\cite{C}. In 2012  Heldt, Knauer and Ueckerdt~\cite{D} improved the
 bound  to $4$ and  also showed that  $2$ is an upper bound  on 
the bend  number of outerplanar graphs. 
 \c{C}ela and Gaar~\cite{CelaGaarEPGOuterplanar}
 showed recently  that  $2$ is also an upper bound on the  monotonic bend number of
outerplanar graphs. Moreover, they give a full characterization of any 
maximal 
outerplanar graph and any cactus\footnote{A connected graph is called a cactus 
iff
 any  two simple 
 cycles in it  share at most one vertex.} 
with (monotonic) bend number equal to $0$, $1$ and $2$ in terms of forbidden induced subgraphs.

Also other graph classes were considered. 
Recently Francis and Lahiri~\cite{KHalinGraphs} proved that  Halin graphs are
in $B_2^m$
and Deniz, Nivelle, Ries and Schindl~\cite{SSplitGraphsB1} 
provided a characterization of split graphs  for which there exists a $B_1$-EPG representation which
uses only L-shaped paths on the grid, i.e.\ paths consisting of a vertical
top-bottom
segment followed by a horizontal left-right segment.

Another line of research on EPG graphs concerns the mutual relationship between the
classes $B_k$ and the classes $B_\ell^m$. Our paper is a contribution in this
direction.   
The chains of inclusions  $B_{0} \subseteq B_{1} \subseteq B_{2} \subseteq
\dots$ and $B_{0}^{m} \subseteq B_{1}^{m} \subseteq B_{2}^{m} \subseteq \dots$
trivially hold. Furthermore, $B_{0} = B_{0}^{m} \subseteq B_{1}^{m}$
and 
 $B_{k}^{m} \subseteq B_{k}$, for every $k$,  are obvious.
In~\cite{F} Heldt, Knauer and Ueckerdt dealt with the question whether the
complete bipartite graph $K_{m,n}$ on $m$ and $n$ vertices in the two
partition classes is in $B_k$. 
They identified several sufficient  conditions which have to be fulfilled by 
$m$, $n$
and $k$  to guarantee that      $K_{m,n}$ is in $B_{k}$ or $K_{m,n}$
is not  in $B_{k}$. They  used this kind of  results to prove that $B_{k} 
\subsetneqq B_{k+1}$ holds for every $k \geqslant 0$.
In this paper, we will derive new results of this type, especially for the monotonic case. 
It is still not known whether  $B_{k}^m \subsetneqq B_{k+1}^m$ also holds.

The relationship between $B_{k}$ and $B_{k}^{m}$ has already been considered in 
the literature. 
Golumbic, Lipshteyn and Stern~\cite{startpaper} conjectured that $B_{1}^{m}
\subsetneqq B_{1}$, which was confirmed in~\cite{E}.
In this paper, we show that $B^m_{k}\subsetneqq B_{k}$ also holds for  $k \in 
\{2,3,5\}$ and $k\geqslant 7$, while the cases $k=4$ and $k=6$ remain open.

Furthermore, we are interested in  the gap  between the  bend number $b(G)$ 
and the monotonic bend number $b^m(G)$ of a graph. More precisely we pose the 
question whether there exists a function $f\colon \nz\to \nz$ such that   $b^m(G)\leqslant f(b(G))$
holds for every graph $G$. As a first step towards answering this question we
show that $B_k\not\subseteq B_{2k-9}^m$ holds for any $k\in \nz$, $k\geqslant 5$, 
which  implies the existence  of  graphs for which
$b^m(G)\geqslant 2k-8$ and $b(G)\leqslant k$,  for any $k\in \nz$, $k\geqslant 5$. 
Moreover,  we show that $b(G)\leqslant 1$ implies $b^m(G)\leqslant 3$.   
\medskip

The rest of the paper is organized as follows.
 \autoref{sec:KmnBkm} deals with  the (monotonic) bend number of $K_{m,n}$. 
First we review some results from the literature on the bend number of
 $K_{m,n}$, where $m\leqslant n$.
In particular, we discuss a  theorem from \cite{F}  and point
 out that the proof of the theorem does not work out for $m=4$ and $m=5$. 
 Further,  we show that the statement of the theorem holds for $m=4$, while  we don't
 know whether it holds for $m=5$. However,  we only exploit the statement of the
 theorem for $m\geqslant 7$ in our later work. 
 In \autoref{sec:LowerBoundLemmas}, we derive two inequalities on $m$, $n$ and 
$k$ 
which have to be fulfilled if $K_{m,n}$ is in $B_{k}^{m}$.
 In \autoref{sec:LargestUpperBoundKmn} we  show that
 $b^m(K_{m,n})\leqslant 2m-2$ for every $m,n \in \nz$, $m\leqslant n$,
 which implies that $b^m(G)\leqslant 2m-2$ holds for every graph $G$ that is an 
 induced 
 subgraph of $K_{m,n}$.
  Moreover,
 we show  that this upper
 bound on $b^m(K_{m,n})$ is best possible, i.e.\  for each
 $m\in \nz$ there  exists an $n_m\in \nz$, $n_m\geqslant m$, such that 
$b^m(K_{m,n_m})=2m-2$. An analogous behavior of $b(K_{m,n})$ has been already
shown in literature (see~\cite{F}).
However, we will see that this maximum bend number is attained already 
for smaller values of $n_m$ in the monotonic case.

In \autoref{sec:B2B2m}, we present a graph which is in $B_{2}$ and not in 
$B_{2}^{m}$ in order to prove  $B_{k}^{m}  \subsetneqq B_{k}$ for $k = 2$. 
In \autoref{sec:B5B5m}, we use the results of \autoref{sec:LowerBoundLemmas} to
prove that $B_{k}^{m}  \subsetneqq B_{k}$ also 
for $k\in\{3,5\}$ and $k\geqslant 7$, 
thus answering  an open question posed in \cite{startpaper} for almost all values of 
$k$.

Finally,  in \autoref{sec:BkBlm} we  investigate the relationship between 
$B_k$ and $B_\ell^m$ for $\ell > k$. 
In \autoref{sec:BkBlmWith2kMinus9} we show  that for odd $k \geqslant 5$ 
there is a graph in $B_{k}$ which is not in $B_{2k-8}^{m}$ and 
for even $k \geqslant 5$ there is a graph in $B_{k}$ which is not in 
$B_{2k-9}^{m}$. 
Then  in \autoref{sec:B1B3m} we prove that $B_{1} \subseteq 
B_{3}^{m}$, giving the first result of this kind.
We summarize our results and discuss some  open questions in 
\autoref{sec:Conclusions}.
\medskip

{\bf Terminology  and notation.}
Finally,  we settle  the  terminology and the notations used throughout the paper.
The crossings of two \emph{grid lines} are called \emph{grid points}. The part 
of a 
grid line between two consecutive grid points is called a \emph{grid edge}.
A grid edge can be \emph{horizontal} or \emph{vertical}.

A \emph{path} on a grid consists of two  grid points, called the \emph{end
points} of the path, 
and a number  of  consecutive grid edges connecting the end points. 
If the two end-points lay on different vertical grid lines, we call the 
left-most point the \emph{start point} and the other one the \emph{terminal 
point}. Otherwise, we call the lower point the \emph{start point} and the other 
one the \emph{terminal point}.
A turn of a path on the grid is called \emph{bend}
and a grid point, at which the path turns, 
is called a \emph{bend point}.
 
The part of a path between two consecutive bend points is called a 
\emph{segment}. 
Also the part of the path from the start point to the first bend point is
called a 
\emph{segment}.  
This is called the \emph{first segment} of the path.
Analogously,  the 
part of the path from the last bend point to the terminal  point is also  called 
a \emph{segment}.
This is the \emph{last segment} of the path.
We consider the intermediate segments in their natural order: 
the segment of the path following the first one is the \emph{second segment}, 
and so on.

The grid points contained in   a segment of a path which are neither
bend points nor end points of that path build the
\emph{interior} of that segment.   
Clearly any  segment consists either  entirely of horizontal grid edges  or
entirely of  vertical grid edges. We call such segments \emph{horizontal} and \emph{vertical segments}, 
respectively. 
Paths without bends correspond to (horizontal or vertical)
segments.

We say that two paths  on a grid \emph{intersect}, 
if they  have at least one common grid edge.
If two segments $S_1$, $S_2$ lie on  the same grid line but do not intersect
(if considered as paths), 
then we call them \emph{aligned}; such a  pair $(S_1, S_2)$ is called    an 
\emph{alignment}. Figure~\ref{fig:possibleACP}(a) depicts two aligned
segments $S_1$ and $S_2$.  

A pair  $(S_1, S_2)$ of  segments is called a
\emph{crossing} if one of the two segments lies on a horizontal grid line,
 the other segment lies on a vertical grid line, 
 and 
 there is a grid point which belongs to the interior of
both  segments. 
Figure~\ref{fig:possibleACP}(b) depicts
a crossing   $(S_1,S_2)$ with  grid point $x$ belonging to  the interior of both
segments.     

A pair  $(S_1, S_2)$ of  segments is called a
\emph{pseudocrossing} if one of the two segments lies on a horizontal grid 
line,
the other segment lies on a vertical grid line, 
and 
there is no grid point which belongs to the interior of
each of the segments. 
Figure~\ref{fig:possibleACP}(c)-(e)  depict different 
pseudocrossings. 

Given a set ${\cal P}$ of pairwise non-intersecting paths on a grid we define 
the
alignments (crossings, pseudocrossings) of  ${\cal P}$ as the set of all
alignments (crossings,  pseudocrossings) $(S_1,S_2)$ for which there exist two 
distinct 
paths $P_1,P_2\in {\cal P}$ such that $S_i$ is a segment of $P_i$,  for $i\in
\{1,2\}$. Figure~\ref{fig:possibleACP}(f) depicts two paths $P_1$ and 
$P_2$
containing two alignments (a horizontal one and a vertical one)
and two pseudocrossings.

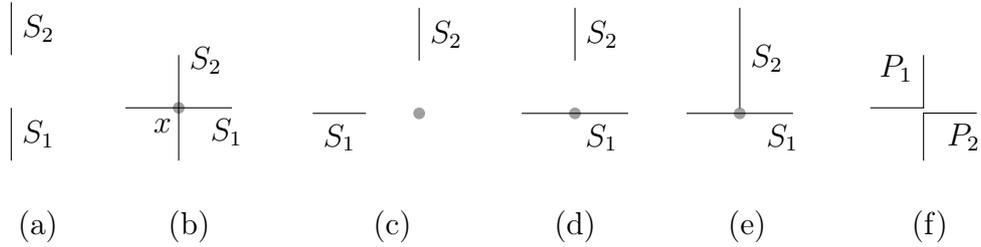
\begin{figure}[ht]
    \centering
     \begin{minipage}[b]{0.06\linewidth}
        \centering
        \begin{center}
\begin{tikzpicture}
  [scale=.7]
  
  \draw (2,1) -- (2,2) node[right,pos=0.5] {$S_{1}$};  
  \draw (2,3) -- (2,4) node[right,pos=0.5] {$S_{2}$};

  \end{tikzpicture}
\end{center}
        (a)
    \end{minipage}
    \quad
    \begin{minipage}[b]{0.15\linewidth}
        \centering
        \begin{center}
\begin{tikzpicture}
  [scale=.7]

  \node[vertexSmallGray] (v) at (2,2) {};
  \draw (1,2) -- (3,2) node[below,pos=0.95] {$S_{1}$};  
  \draw (2,1) -- (2,3) node[right,pos=0.95] {$S_{2}$};  
  \node[left,below] at (1.7,2) {$x$};

\end{tikzpicture}
\end{center}
        (b)
    \end{minipage}
    \quad 
    \begin{minipage}[b]{0.16\linewidth}
        \centering
        \begin{center}
\begin{tikzpicture}
  [scale=.7]
  
  \draw (0,1) -- (1,1) node[below,pos=0.5] {$S_{1}$};  
  \draw (2,2) -- (2,3) node[right,pos=0.5] {$S_{2}$};  
   \node[vertexSmallGray] (v) at (2,1) {};
  
  \end{tikzpicture}
\end{center}
        (c)
    \end{minipage}
    \quad
    \begin{minipage}[b]{0.11\linewidth}
        \centering
        \begin{center}
\begin{tikzpicture}
  [scale=.7]
  
   \node[vertexSmallGray] (v) at (2,1) {};
  \draw (1,1) -- (3,1) node[below,pos=0.75] {$S_{1}$};  
  \draw (2,2) -- (2,3) node[right,pos=0.5] {$S_{2}$};  
  
  \end{tikzpicture}
\end{center}
        (d)
    \end{minipage}
    \quad
    \begin{minipage}[b]{0.14\linewidth}
        \centering
        \begin{center}
    \begin{tikzpicture}
    [scale=.7]
    
    \node[vertexSmallGray] (v) at (2,1) {};
    \draw (1,1) -- (3,1) node[below,pos=0.9] {$S_{1}$};  
    \draw (2,1.1) -- (2,3) node[right,pos=0.5] {$S
      _{2}$};  
    
    \end{tikzpicture}
\end{center}
        (e)
    \end{minipage}
    \quad
    \begin{minipage}[b]{0.13\linewidth}
        \centering
        \begin{center}
\begin{tikzpicture}
  [scale=.7]
  
  \draw (1,2) -- (2,2) -- (2,3) node[left,pos=0.75] {$P_{1}$};  
  \draw (2,1) -- (2,1.9) -- (3,1.9) node[below,pos=0.75] {$P_{2}$};

  \end{tikzpicture}
\end{center}
        (f)
    \end{minipage}

    \caption{(a) An alignment $(S_1,S_2)$. 
        (b) A crossing $(S_1,S_2)$.  
        (c)-(e) Different  pseudocrossings $(S_1,S_2)$.
        (f) Two  paths $P_1$ and $P_2$ containing two alignments
        and two pseudocrossings. }
    \label{fig:possibleACP}
\end{figure}

Finally, notice that in an EPG representation of a graph $G$ with vertex set $V$ we will 
denote by $P_v$ the  path on the grid corresponding to the vertex $v \in V$.


\section{Complete Bipartite Graphs}
\label{sec:KmnBkm}

The aim of this section is to summarize existing results on the $B_k$-EPG
representation of complete bipartite graphs and 
derive 
new upper and lower 
bounds on their (monotonic) bend number.
We start by investigating some  results from the literature in 
\autoref{sec:UpperBoundsBendNumberKmn}.
Then we derive two Lower-Bound-Lemmas in \autoref{sec:LowerBoundLemmas}.
 Eventually, in Section 2.3, we give an upper bound
 on the monotonic bend number of $K_{m,n}$ for every $m, n \in \nz$, 
 $m\leqslant n$.
 The results obtained in this section will be used in \autoref{sec:B5B5m},
where 
the relationship between  $B_{k}^{m}$ and $B_{k}$ for $k\geqslant 3$ is 
investigated.

Throughout this section we consider the complete bipartite graph $K_{m,n}$ with 
$m\leqslant n$. We denote the two partition classes of
 $K_{m,n}$ by $A$ and $B$, where $|A|=m$ and $|B|=n$.
In an EPG representation we denote the set of all paths that 
correspond to vertices of $A$ and $B$ by $\mathcal{P}_A$ and $\mathcal{P}_B$, 
respectively;   
so $\mathcal{P}_A = \{P_v: v \in A\}$ and  $\mathcal{P}_B = \{P_w: w \in B\}$.

\subsection{Upper Bounds on the Bend Number}
\label{sec:UpperBoundsBendNumberKmn}
First of all notice that the bend number of $K_{m,n}$ for $m\in \{0,1,  2\}$ is
known. The trivial case $m=0$  corresponds to a graph without any edges and 
hence
$b(K_{0,n})=b^m(K_{0,n})=0$, for all $n\in \nz$. 

The other trivial case  $m=1$
corresponds to a star graph with $n+1$ vertices.  A $B_0$-EPG representation of
this graph  consists of a horizontal path $P$ with $n$ grid
edges to  represent the central vertex, and the pairwise different grid edges of
$P$ represent the other vertices. 
Thus $b(K_{1,n})=b^m(K_{1,n})=0$,
for all $n\in \nz$.

The bend number of $K_{2,n}$ has been determined by Asinowski and Suk~\cite{B} 
for all
$n\in \nz$: $b(K_{2,n})=2$ if and only if $n\geqslant 5$, $b(K_{2,n})=1$ if and 
only if $2\leqslant 
n\leqslant 4$, and
$b(K_{2,n})=0$ if and only if $n\leqslant 1$.
The EPG representations for $K_{2,n}$ given in~\cite{B} are monotonic, therefore 
$b^m(K_{2,n}) = b(K_{2,n})$ holds for all $n \in \nz$.

The more general case $m\geqslant 3$ has been considered by Heldt, Knauer and
Ueckerdt in \cite{F}. We first 
discuss the following  result of these authors. 

\begin{theorem}[Heldt, Knauer, Ueckerdt~\cite{F}]
    \label{KmninBmMinus1}
    If $m \geqslant 4$ is even and  $n = \frac{1}{4}m^3 - \frac{1}{2}m^2 - m + 4$, then $K_{m,n}$ is in $B_{m-1}$ but not in 
    $B_{m-2}$.
    If $m \geqslant 7$ is odd and $n=\frac{1}{4}m^3 - m^2 + \frac{3}{4}m$, then 
    $K_{m,n}$ is in $B_{m-1}$ but not in 
    $B_{m-2}$. 
  \end{theorem}
  The above theorem makes no statement for the cases $m=3$ and $m=5$. 
However,  in~\cite{F} the authors  claim that the statement for odd $m$    holds also for 
$m=5$ (see~\cite[Theorem 4.4.]{F}). 
But the proof provided in~\cite{F} is not correct for $m=5$ and  we
 do not know whether the statement  is  true  in this case. 
Also  for the case $m=4$
the proof provided in~\cite{F} is not correct, however  in this case the statement is
true as argued below.

To be more precise, in~\cite{F}
on the one hand  
the authors provide a  $B_{m-1}$-EPG representation for 
$K_{m,n}$ for  $m \geqslant 3$ and $n$ defined as in \autoref{KmninBmMinus1},
i.e.\ a constructive proof for one part of~\cite[Theorem 4.4.]{F}.
On the other hand
the 
Lower-Bound-Lemma I~\cite[Lemma~4.1]{F} is used in order to show that $K_{m,n}$ 
is
not in  $B_{m-2}$ for $n$ defined as in  \autoref{KmninBmMinus1}.
This Lower-Bound-Lemma I states that 
$$(k+1)(m+n) \geqslant mn + \sqrt{2k(m+n)}$$ 
holds for every $B_{k}$-EPG representation of $K_{m,n}$ with 
$n\geqslant m \geqslant 3$. 
Further they observe  that for  $n$ defined as in  \autoref{KmninBmMinus1} the
inequality $n \geqslant (m-1)^{2}$ holds,  while   the inequality of 
the Lower-Bound-Lemma I is not fulfilled for $n\geqslant (m-1)^{2}$ and $k = 
m-2$, thus negating the membership of the corresponding graphs in $B_{m-2}$.
However,  for  $n$ defined as in  \autoref{KmninBmMinus1}, the inequality
$n\geqslant (m-1)^{2}$ holds only if  $m\geqslant 6$. Thus, the proof provided
for~\cite[Theorem 4.4]{F}   only works for $m\geqslant 6$.

For $m = 4$ we have $n=8$, and the construction in \cite{F}   proves that $K_{4,8}$ is in 
$B_{3}$. Furthermore, by applying the Lower-Bound-Lemma I  for $m=4$, $n=6$  
and 
$k=2$ we get that $K_{4,6}$ is not in $B_{2}$. This implies that also   $K_{4,8}$ is not in $B_{2}$.
Therefore, the statement of \autoref{KmninBmMinus1} is also true  for $m=4$.

If $m=5$  the construction in \cite{F}  yields that $K_{5,10}$ is in $B_{4}$. If we use 
the Lower-Bound-Lemma I, then we get that $K_{5,11}$ is not in $B_{3}$ and that 
the bend number of $K_{5,10}$ is at least 3. Therefore, the bend number of 
$K_{5,10}$ could be either 3 or 4.

\subsection{Lower-Bound-Lemmas}
\label{sec:LowerBoundLemmas}
In order to investigate the relationship between $B_{k}^{m}$ and $B_{k}$ for 
large values of $k$, we first derive a Lower-Bound-Lemma for $B_{k}^{m}$-EPG 
representations similarly to the Lower-Bound-Lemma I  for $B_{k}$-EPG 
representations from \cite{F}. To this end, we use an  auxiliary result from
\cite[Lemma 4.6]{F}. 
\begin{lemma}[Heldt, Knauer, Ueckerdt~\cite{F}]
    \label{lbl}
    Let $3 \leqslant m \leqslant n$.  Consider $K_{m,n}$ and denote by $A$ the
    subset of vertices  of cardinality $m$ in the partition of the vertex set
    of $K_{m,n}$.
    Consider further  a $B_{k}$-EPG representation of 
    $K_{m,n}$ and denote by  ${\cal
      P}_A$ be the set of the paths on the grid corresponding to the vertices of $A$ in this representation.
    Let $c$ be  the  total number of crossings of ${\cal P}_A$. Then, the
    following inequality holds: 
    \begin{align*}
    n(2m-k-2) \leqslant 2c + 2(k+1)m.
    \end{align*}
\end{lemma}
In the following we  derive inequalities on $m$, $n$ and $k$ which  hold
whenever a $K_{m,n}$ is in $B_{k}^{m}$. 
The following lemma is a first step towards such a result.
Note that  $\lfloor x \rfloor$
is the greatest integer less than or equal to $x$ 	
and $\lceil x \rceil$ is the least integer greater than or equal to $x$ for any 
real number $x$.
\begin{lemma}
    \label{condition_k3n_in_B3}
    Let $3 \leqslant m \leqslant n$. 
    Consider $K_{m,n}$ and denote by $A$ the
    subset of vertices  of cardinality $m$ in the partition of the vertex set
    of $K_{m,n}$. Consider further  a $B_{k}$-EPG representation of $K_{m,n}$
    and denote by    ${\cal
      P}_A$ be the set of the paths on the grid corresponding to the vertices of $A$ in this representation.
    Let   $a$, $c$ and $p$  be the total number of
    alignments, crossing and pseudocrossings  of  ${\cal P}_A$,
    respectively.  
    Then, the     following inequality holds: 
    \begin{align*}
    n\left(m - \left\lceil \frac{k+1}{2}\right\rceil\right) \leqslant a + 2c +
      p.
    \end{align*}   
\end{lemma}
\begin{proof}  
    Let $w$ be a vertex of $B$.
    For each vertex $v \in A$ 
    we denote by $e_{v}^w$ a fixed but arbitrarily chosen common grid edge of 
    $P_v$ and $P_w$. 
    Such an edge exists, because $P_w$ intersects $P_v$  since $w$ is 
    adjacent to all vertices of $A$. The grid edges $e_{v}^w$ 
    for all  $v \in A$
    are pairwise disjoint, because the vertices of $A$   are not adjacent to 
    each other.
    
    We order the vertices      $A = \{v_1, \dots, v_m\}$ in such a way that
    $e_{v_i}^w$ precedes $e_{v_{i+1}}^w$ in the path $P_w$, for all $i\in \{1,2,\ldots, m-1\}$. 
    Let $x_w$, $y_w$ and $z_w$ be the number of indices $i \in \{1, \dots, m-1\}$ such that 
    $e^w_{v_i}$ and $e^w_{v_{i+1}}$ lie on the same segment of $P_w$, on consecutive      segments of $P_w$, and
    neither on the same nor on consecutive segments of $P_w$,  respectively.       
    Then, clearly $x_w + y_w + z_w = m-1$ holds.
  
    If $e^w_{v_i}$ and $e^w_{v_{i+1}}$ lie neither on the 
    same nor on consecutive segments of $P_w$,  
    then the subpath of $P_w$ between (and not including) the two segments of 
    $P_w$  containing $e^w_{v_i}$ and
    $e^w_{v_{i+1}}$ contains at least one segment and does not contain any
    $e^w_{v_{i'}}$ for $i' \in \{1, \dots, m\}$.
    Let us call such a subpath a \emph{free subpath} of $P_w$. 
    Since 
    $P_w$ has at most $k+1$ segments and each free subpath is preceded and
    also 
    succeeded by a
    segment containing 
    $e^w_{v_{i'}}$ for some $i' \in \{1, \dots, m\}$, 
    the number of free subpaths 
     is at most $\left\lfloor \frac{k}{2}\right\rfloor$
  and  hence      $z_w \leqslant \left\lfloor \frac{k}{2}\right\rfloor$ holds.
        To summarize up to now we have shown that 
    \begin{align}
    \label{eq:InProofacpxyz}
    m - \left\lceil \frac{k+1}{2}\right\rceil
    = m - 1 - \left\lfloor \frac{k}{2}\right\rfloor 
    \leqslant m - 1 - z_w 
    = x_w + y_w 
    \end{align}
    holds.

    It remains to determine an upper bound on $x_w + y_w $.   Towards this end, 
  let   $S^w_i$ be the segment of $P_{v_i}$ that contains $e_{v_i}^w$ for $i \in \{1,2\ldots,m\}$.
Now  we consider the pairs $(S_i^w,S_{i+1}^w)$, $i \in \{1,2,\ldots, m-1\}$.

We  denote by  $a_w$ the number of indices $i \in \{1, \dots, m-1\}$ such that 
$e^w_{v_i}$ and $e^w_{v_{i+1}}$ lie on the same segment of $P_w$ 
and the pair $(S_i^w,S_{i+1}^w)$ is an alignment. 
It is easy to see that if $e^w_{v_i}$ and $e^w_{v_{i+1}}$ lie on the same 
segment of $P_w$, then the corresponding  segments $S^w_i$ and  $S^w_{i+1}$  of $P_{v_i}$ and $P_{v_{i+1}}$ 
lie on the same grid line and therefore $(S_i^w,S_{i+1}^w)$ is an alignment. 
Thus, $a_w = x_w$  holds.

Furthermore,  let   $c_w$ ($p_w$) denote the number of $i \in \{1, \dots, m-1\}$
such that $e^w_{v_i}$ and $e^w_{v_{i+1}}$ lie on  consecutive segments of $P_w$
and the pair $(S_i^w,S_{i+1}^w)$ is a crossing (pseudocrossing). 
If $e^w_{v_i}$ and $e^w_{v_{i+1}}$ lie on consecutive segments of $P_w$, 
then one of the corresponding  segments $S^w_i$ and  $S^w_{i+1}$  is
horizontal and the other one is vertical.
Hence $(S^w_i, S^w_{i+1})$ is either a crossing or a pseudocrossing.
Therefore,  $c_w + p_w = y_w$ holds.

As a result, we can use~\eqref{eq:InProofacpxyz} to deduce that
    \begin{align*}
m - \left\lceil \frac{k+1}{2}\right\rceil
\leqslant  x_w + y_w 
= a_w + c_w + p_w
\end{align*}
holds. 
Summing this up over all vertices $w \in B$ yields
\begin{align*}
n\left(  m - \left\lceil \frac{k+1}{2}\right\rceil \right)
\leqslant \sum_{w\in B}(a_w+c_w+p_w).
\end{align*}

    It remains to determine an upper bound on $\sum_{w\in B} (a_w + c_w + p_w)$.
    Towards this end, let $a_B =\sum_{w \in B} a_w$, $c_B =\sum_{w \in B} c_w$ 
    and $p_B =\sum_{w \in B} p_w$.    Clearly, an alignment (crossing, pseudocrossing) $(S_i^w,S_{i+1}^w)$,
    for $w \in B$ and for $i \in     \{1,2,\ldots, m-1\}$, is an alignment
    (crossing, pseudocrossing) of ${\cal P}_A$, since $S_i^w$ is a segment of
    $P_{v_i}$, for  $i \in \{1,2,\ldots,m\}$.    
    This implies that $a_B\leqslant a$ and  $p_B\leqslant p$ because the 
    alignments and 
    pseudocrossings counted in $a_B$ and $p_B$ are pairwise distinct
    due to the fact that the 
    paths in ${\cal P}_A$ are 
    pairwise non-intersecting and also the paths in 
    ${\cal P}_B$ 
    are pairwise non-intersecting.
    
    The crossings counted in $c_B$ are not necessarily pairwise distinct 
    because a crossing 
    $(S_i^w,S_{i+1}^w)$ can also appear as a crossing
    $(S_j^{w'},S_{j+1}^{w'})$, for some $w, w' \in B$, $w\neq w'$ and some  
    $i,j \in 
    \{1,2,\ldots,m-1\}$, see \autoref{fig:2CountCrossing2}. (Notice that in this
    case the vertices $v_i$ and $v_j$ coincide.)
  However, the same crossing cannot be 
    counted more than twice in $c_B$ because the paths in ${\cal P}_B$ are 
    pairwise 
    non-intersecting, so $c_B\leqslant 2c$ holds.
     
    \begin{figure}[ht]
        \centering
        \begin{center}
\begin{tikzpicture}
  [scale=.7]
  
  \draw (1,2.1) -- (1.9,2.1) -- (1.9,3) node[left,pos=0.75] {$P_{w}$};  
  \draw (2.1,1) -- (2.1,1.9) -- (3,1.9) node[below,pos=0.75] {$P_{w'}$};  
  \draw [dashed] (1,2)-- (3,2) node[below,pos=-0.4]{$S_{i}^w=S_j^{w'}$};
  \draw [dashed] (2,1) -- (2,3) node[right,pos=1]{$S_{i+1}^w=S_{j+1}^{w'}$};
  
  \end{tikzpicture}
\end{center}
        \caption{The crossings $(S_i^w,S_{i+1}^w)$ and
            $(S_j^{w'},S_{j+1}^{w'})$
            coincide.}
        \label{fig:2CountCrossing2}
    \end{figure}
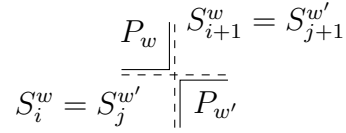
    
    Finally,  we can deduce
    \begin{align*}
    n\left(  m - \left\lceil \frac{k+1}{2}\right\rceil \right)
    \leqslant \sum_{w\in B}(a_w+c_w+p_w)
    = a_B + c_B + p_B
    \leqslant a + 2c + p.
    \end{align*}
   \end{proof}

The next lemma  gives bounds on the number of alignments, crossings and 
pseudocrossings.

\begin{lemma}
    \label{notTooManyCrossings}
    Consider two paths $P_{1}$, $P_{2}$ in a $B_{k}$-EPG representation that do 
    not intersect (i.e.\ have no grid edge in common). Let $a$, $c$ and $p$ be the number of alignments, crossings 
    and pseudocrossings of $\{P_{1}, P_{2}\}$, respectively. If one path starts 
    horizontally and the other one starts vertically, then 
    \begin{enumerate}[label=(\alph*)]
        \setlength{\itemsep}{0pt}
        \item \label{boundcPlusp} $\displaystyle c + p \leqslant 2\left\lfloor 
        \frac{k+1}{2}\right\rfloor \left\lceil \frac{k+1}{2}\right\rceil + 
        \left\lceil \frac{k+1}{2}\right\rceil -  \left\lfloor 
        \frac{k+1}{2}\right\rfloor$ and
        \item \label{boundaPlusc} if the paths are monotonic 
        $\displaystyle a + c \leqslant k+1$  holds.
    \end{enumerate}
    If both paths start horizontally or both paths start vertically, then
    \begin{enumerate}[label=(\alph*), resume]
        \setlength{\itemsep}{0pt}
        \item \label{boundcPluspBetter} $\displaystyle c + p \leqslant 
        2\left\lfloor \frac{k+1}{2}\right\rfloor \left\lceil 
        \frac{k+1}{2}\right\rceil$ and        
        \item \label{boundaPluscBetter} if the paths are monotonic 
        $\displaystyle a + c \leqslant k$ holds.
    \end{enumerate}
    
\end{lemma}
\begin{proof}
    First we consider~\ref{boundcPlusp} and~\ref{boundcPluspBetter}. In a  crossing 
    or a pseudocrossing $(S_1,S_2)$ of  $\{P_{1},P_{2}\}$ one of the segments
    is horizontal and the other one is vertical. 
    Notice that a path that starts with a 
    horizontal segment  has at most $\left\lceil \frac{k+1}{2}\right\rceil$ 
    horizontal and at most 
    $\left\lfloor \frac{k+1}{2}\right\rfloor$ vertical segments, whereas  a path
    that starts with a  vertical segment 
     has at most $\left\lfloor \frac{k+1}{2}\right\rfloor$ horizontal and at 
     most
    $\left\lceil \frac{k+1}{2}\right\rceil$ vertical segments. 
    If one of the paths starts horizontally and the other path starts 
    vertically this 
    implies 
    that 
    \begin{align*}
    c + p &\leqslant \left\lfloor \frac{k+1}{2}\right\rfloor^{2} + \left\lceil 
    \frac{k+1}{2}\right\rceil^{2}\\
    &= 2\left\lfloor \frac{k+1}{2}\right\rfloor \left\lceil 
    \frac{k+1}{2}\right\rceil + \left(\left\lceil \frac{k+1}{2}\right\rceil -  
    \left\lfloor \frac{k+1}{2}\right\rfloor\right)^2\\
    &= 2\left\lfloor \frac{k+1}{2}\right\rfloor \left\lceil 
    \frac{k+1}{2}\right\rceil + \left\lceil \frac{k+1}{2}\right\rceil -  
    \left\lfloor \frac{k+1}{2}\right\rfloor, 
    \end{align*}
    where we can omit the square because the squared value is either $0$ or 
    $1$, and hence~\ref{boundcPlusp} holds. With the same arguments we obtain
    \begin{align*}
      c + p \leqslant 2\left\lfloor \frac{k+1}{2}\right\rfloor \left\lceil
      \frac{k+1}{2}\right\rceil
    \end{align*}
    for paths that start in the same direction. Thus~\ref{boundcPluspBetter} 
    is satisfied.   
    
    Next we consider~\ref{boundaPlusc}, so assume the paths are monotonic. It 
    is easy to see that each segment of $P_{1}$ cannot cross two or more 
    segments of $P_{2}$ and cannot be aligned with two  or more segments of 
    $P_{2}$. Furthermore, whenever a segment of $P_{1}$  
    crosses a segment of 
    $P_{2}$, it cannot be aligned with another segment of $P_{2}$.
    Moreover, whenever a segment of $P_{1}$ is aligned with a
    segment of $P_{2}$, it cannot cross another segment of $P_{2}$. Hence each 
    segment of $P_{1}$ can be 
    part of at most one crossing or alignment. This implies~\ref{boundaPlusc} 
    as $P_{1}$ has at most $k+1$ segments.

    In order to prove~\ref{boundaPluscBetter} assume without loss of generality 
    that both paths start horizontally. The arguments of~\ref{boundaPlusc} 
    imply  that each segment of each of the paths can appear in  at most one 
    crossing or one alignment. We distinguish two cases.
    If one of the paths starts in a lower grid line than the other, then the 
    first segment of this path can neither be aligned to nor cross the other 
    path. Therefore, alignments and crossings can only occur on the remaining 
    $k$ segments of the path and hence $a+c \leqslant k$ holds.
    If both paths start on the same grid line, then let without loss of
    generality the first segment of $P_{1}$ lie to the left of the first
    segment of $P_2$. It is easy to see that the second segment of $P_{1}$ can
    neither be aligned to nor cross any segment of $P_{2}$.
    Therefore, also in this case we have $a+c \leqslant k$. This 
    proves~\ref{boundaPluscBetter}.
\end{proof}

Next we   combine the bounds on the number of crossings derived in 
\autoref{notTooManyCrossings} with \autoref{lbl} in the following result.
\begin{lemma}
    \label{mlbl}
    Let $3 \leqslant m \leqslant n$. In every $B_{k}^{m}$-EPG representation of 
    $K_{m,n}$ 
    \begin{align*}
    n(2m-k-2) \leqslant k(m-1)m + \frac{1}{2}m^2 + 2(k+1)m
    \end{align*}
    holds.
\end{lemma}
\begin{proof}
    Let $c$ denote the number of crossings of the paths in 
    ${\cal P}_A$. 
    Every $B_{k}^{m}$-EPG  representation is also a $B_{k}$-EPG representation. 
    Therefore, it follows   from \autoref{lbl} that
    \begin{align}
    \label{lblinequality}
    n(2m-k-2) \leqslant 2c + 2(k+1)m
    \end{align}
    holds for every $B_{k}^{m}$-EPG representation of $K_{m,n}$. Now we give
    an upper bound on $c$. Let $\ell$ be the number of paths in ${\cal P}_A$ 
    which start with a horizontal segment. Then, $m - \ell$ paths of ${\cal
      P}_A$   start with a  vertical segment. Since the paths in ${\cal P}_A$
    are pairwise non-intersecting, the number $c$ of  crossings of
    ${\cal P}_A$ can be calculated as $c=\sum_{\{v,v'\}\subseteq A}
    c_{v,v'}$,  where
    $c_{v,v'}$ is the number of crossings of $\{P_v, P_{v'}\}$. 
    
    If both $P_v$ and $P_{v'}$  start with a horizontal (vertical) segment,
    then $c_{v,v'}\leqslant k$
      by \autoref{notTooManyCrossings}\ref{boundaPluscBetter}. If one of the 
      paths
     $P_v$ and $P_{v'}$  starts with a horizontal  segment and the other one
     starts with a vertical  segment, then $c_{v,v'}\leqslant k+1$ by 
    \autoref{notTooManyCrossings}\ref{boundaPlusc}. 
    Notice that there  are exactly $\ell(m-\ell)$ pairs of paths  $P_v$ and $P_{v'}$ 
    with the latter property and $\binom{m}{2}-\ell(m-\ell)$ pairs of  paths
    $P_v$ and $P_{v'}$  both starting with a horizontal (vertical) segment. 
    In total we get 
    \begin{align*}
    c=\sum_{\{v,v'\}\subseteq A }
    c_{v,v'}\leqslant k \binom{m}{2} + \ell(m-\ell).
    \end{align*}
    Since  $\ell(m-\ell) \leqslant 
    \left(\frac{m}{2}\right)^{2}$ for all $0 \leqslant \ell \leqslant m$ we get 
    \begin{align*}
    c &\leqslant k \binom{m}{2} + \frac{m^2}{4} = \frac{1}{2} 
    \left( k(m-1)m + \frac{1}{2}m^2 \right),
    \end{align*}
   which in combination with   \eqref{lblinequality} completes the proof.
\end{proof}

Next we  combine the bounds on the number of crossings derived in 
\autoref{notTooManyCrossings} and  \autoref{condition_k3n_in_B3} as follows.
\begin{lemma}
    \label{mlbl2}
    Let $3 \leqslant m \leqslant n$. In every $B_{k}^{m}$-EPG representation of 
    $K_{m,n}$ 
    \begin{align*}
    n\left(m - \left\lceil \frac{k+1}{2}\right\rceil\right) \leqslant 
    \binom{m}{2}&\left( 
    2\left\lfloor \frac{k+1}{2}\right\rfloor \left\lceil 
    \frac{k+1}{2}\right\rceil
    + k \right) + \\ 
    & \frac{1}{4}m^2\left(1 + \left\lceil \frac{k+1}{2}\right\rceil -  
    \left\lfloor \frac{k+1}{2}\right\rfloor \right) 
    \end{align*}
    holds.
\end{lemma}
\begin{proof}
    We combine \autoref{condition_k3n_in_B3} and \autoref{notTooManyCrossings} 
    by
    proceeding analogously  as  in the proof of \autoref{mlbl}. 
    
    In particular, 
    let $a$, $c$ and $p$ be the number of alignments, crossings 
    and pseudocrossings of ${\cal P}_A$, respectively.
    As done in the proof of \autoref{mlbl} we can compute $c$ as the sum of the 
    number of crossings $c_{v,v'}$ of $\{P_v,P_{v'}\}$ over all
    pairs $\{v,v'\}\subseteq A$. 
    Similarly we write $p$ and $a$  
    as the sum of the  number of pseudocrossings $p_{v,v'}$
    (alignments $a_{v,v'}$) of $\{P_v,P_{v'}\}$ over all
    pairs $\{v,v'\}\subseteq A$. 
    Thus,  we obtain 
    $a+2c+p=\sum_{\{v,v'\}\subseteq A}
    (a_{v,v'}+c_{v,v'})+\sum_{\{v,v'\}\subseteq A}
    (c_{v,v'}+p_{v,v'})$.
    
    Then, we use  \autoref{notTooManyCrossings}~\ref{boundaPlusc} 
    and~\ref{boundaPluscBetter} to bound
    each
     summand of the first sum  from above and
     \autoref{notTooManyCrossings}~\ref{boundcPlusp} 
     and~\ref{boundcPluspBetter} to bound each
     summand of the second sum  from above. Then we transform the sum of these upper bounds
     analogously as in the proof of  \autoref{mlbl} and  finally  use \autoref{condition_k3n_in_B3}  to
     bound $a+2c+p$ from below. This   completes the proof.
   \end{proof}

To summarize  \autoref{mlbl} and \autoref{mlbl2} provide inequalities on $m$, 
$n$ and $k$ which hold whenever a $K_{m,n}$ with $3 \leqslant m \leqslant n$
is in $B_{k}^{m}$. These inequalities are  used in 
Sections~\ref{sec:LargestUpperBoundKmn} and~\ref{sec:B5B5m}.

\subsection{Upper Bounds on the Monotonic Bend Number}
\label{sec:LargestUpperBoundKmn}

In~\cite{F} a lot of work has been done to determine the bend number of 
$K_{m,n}$ in dependence of $m$ and $n$. In particular, it  was proven 
that  $b(K_{m,n})=2m-2$ for  $m\geqslant 3$ and $n \geqslant
m^4 - 2m^3 + 5m^2 - 4m + 1$.   We  deduce a similar result 
for the monotonic case.

We first generalize a result of~\cite{C}. There it was shown by slightly 
modifying a construction of~\cite{startpaper} that $K_{m,n} \in B_{2m-2}$ for 
all $n$. We modify the construction of~\cite{C} and  give an  analogous  
result for the monotonic case.

\begin{theorem}
    \label{thm:KmnInB2mMinus2}
    It holds that $K_{m,n} \in B_{2m-2}^{m}$.
\end{theorem}
\begin{proof}
    In order to prove this, it is enough to give a $B_{2m-2}^{m}$-EPG 
    representation of $K_{m,n}$, which can be found in 
    \autoref{fig:KmnInB2mMinus2m}. Each  vertex of $K_{m,n}$  
    belonging to  the partition class $A$  of size $m$ is represented in the 
    grid
    by a path consisting
    of just one horizontal 
    segment.
    Each  of the $n$ vertices of the other partition class $B$  is represented 
    in the
    grid by a staircase with 
    $2m-2$ bends. The staircases have  pairwise empty intersections. 
\end{proof}

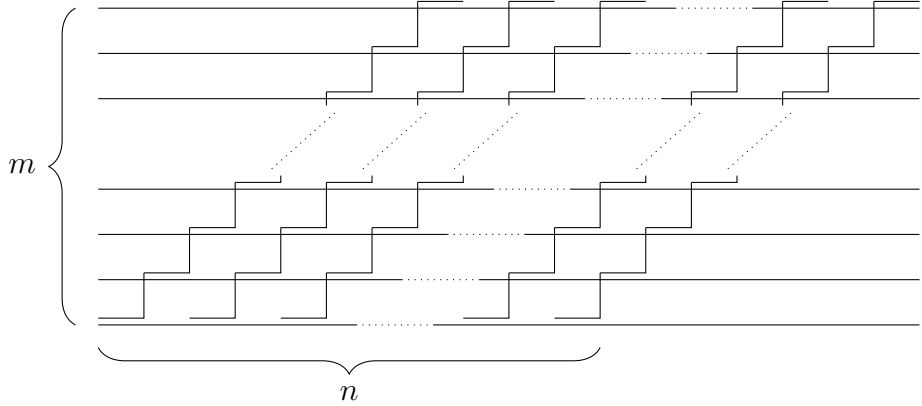
\begin{figure}[ht]
    \begin{center}
\begin{tikzpicture}
  [scale=.60]

\def \gridUp {11}
\def \gridDown {1.7}
\def \gridLeft {2}
\def \gridRight {20}
\def \vspace {0.15}

\foreach \x in {1,2,3,4,6,7,8}
{
	\draw (\gridLeft,\x) -- (\x+6.5+\vspace,\x);
	
	\draw[dotted] (\x+6.5+\vspace,\x) -- (\x+8.5-\vspace,\x);
	
	\draw (\x+8.5-\vspace,\x) -- (\gridRight,\x);		
}

\foreach \x in {1, 3, 5, 9, 11}
{
	\foreach \offset in {1, 2, 3, 6, 7}
	{
		\draw (\x + \offset,\offset+\vspace) -- (\x+1+\offset,\offset+\vspace) -- (\x+1+\offset,\offset+1+\vspace);
	}

}

\def \offset {4}

	\foreach \x in {1, 3, 5, 9, 11}
	{
		\draw (\x+\offset,\offset+\vspace) -- (\x+1+\offset,\offset+\vspace) -- (\x+1+\offset,\offset+2*\vspace);
		
		\draw[dotted] (\x+\offset+0.8,\offset+3*\vspace) -- (\x+\offset+2.2, \offset+2-2*\vspace);
		
	\draw (\x+\offset+2,\offset+2-\vspace) -- (\x+2+\offset,\offset+2 + \vspace);
		
	}

\def \offset {8}

	\foreach \x in {1, 3, 5, 9, 11}
	{
		\draw (\x+\offset,\offset+\vspace) -- (\x+\offset+1,\offset+\vspace);
	}

\draw [decorate,decoration={brace,amplitude=10pt},xshift=0pt,yshift=0pt]
(1.5,1) -- (1.5,8) node [black,midway,xshift=-0.7cm] 
{$m$};  

\draw [decorate,decoration={brace,amplitude=10pt},xshift=0pt,yshift=0pt]
(13,0.5) -- (2,0.5) node [black,midway,yshift=-0.6cm] 
{$n$};  
  
\end{tikzpicture}
\end{center}
    \caption{A $B_{2m-2}^{m}$-EPG representation of $K_{m,n}$.}
    \label{fig:KmnInB2mMinus2m}
\end{figure}

Note that \autoref{thm:KmnInB2mMinus2} implies that $b^m(G)\leqslant 2m-2$ 
holds for every graph~$G$ that is an induced 
subgraph of $K_{m,n}$.
Furthermore,  \autoref{thm:KmnInB2mMinus2}  shows that for fixed $m$ and 
varying $n$, $b(K_{m,n})\leqslant 
b^m(K_{m,n})\leqslant 2m-2$ holds.
Hence,  the upper bound on the number of 
bends needed for an EPG representation of $K_{m,n}$ with $3\leqslant m 
\leqslant n$ is
 the same, namely $2m-2$, no matter whether all kind of bends or only monotonic 
bends are allowed.
This fact is even more surprising if we take into account 
\autoref{KmnInBkNotInB2kMinus89}, which states   the existence of graphs for
which   the gap between the bend number and the monotonic bend number   
can be arbitrarily  large.

However, it turns out that  the upper bound on $b^m(K_{m,n})$ is already 
reached for
a smaller $n$ than the upper bound on $b(K_{m,n})$.
In particular, the above stated result from~\cite{F} implies that 
 $b(K_{m,n})=2m-2$ for $n \geqslant N_{1}$ for some $N_{1} \in \Theta(m^{4})$.
As a consequence of the next result it follows that
$b^m(K_{m,n})=2m-2$ for $n \geqslant N_{2}$ already for some $N_{2} 
\in \Theta(m^{3})$.

\begin{theorem}
    \label{ifTooLargeKmnIsNotInB2mMinus3m}
    Let $3 \leqslant m$. If $n \geqslant 2m^3 - \frac{1}{2}m^2 - m + 1$ then 
    $K_{m,n} \not \in B_{2m-3}^{m}$. 
\end{theorem}
\begin{proof}
    Suppose, in order to derive a contradiction,  that $K_{m,n} \in B_{2m-3}^{m}$. By applying \autoref{mlbl} for 
    $k=2m-3$ we get that
    \begin{align*}
    &&n(2m - (2m-3) - 2) &\leqslant (2m-3)(m-1)m + \frac{1}{2}m^2 + 2(2m-2)m
    \end{align*}
  Then, doing the maths operations we have that  $n \leqslant 2m^3 -
  \frac{1}{2}m^2 - m$ has to hold.  This  contradicts  $n \geqslant 2m^3 -      \frac{1}{2}m^2 - m + 1$.
\end{proof}

\section{Relationship between \texorpdfstring{$B_{k}^{m}$}{Bkm} and
  \texorpdfstring{$B_{k}$}{Bk}}
\label{allBkBkm}
It is an open question of~\cite{startpaper} to determine the relationship 
between $B_{k}^{m}$ and $B_{k}$ for $k \geqslant 1$. 
Obviously $B_{k}^{m} \subseteq B_{k}$ holds for every $k$. 
In~\cite{startpaper} Golumbic, Lipshteyn and Stern conjectured that $B_{1}^{m} 
\subsetneqq B_{1}$. 
This conjecture was confirmed by Cameron, Chaplick and Ho\`{a}ng in~\cite{E} by 
showing that the graph $S_{3}$, 
which was known to be in $B_{1}$ from~\cite{startpaper}, is not in $B_{1}^{m}$. 
The graph  $S_3$ is  isomorphic to the  subgraph induced by the vertices 
$\{a,b,c,d,e,f\}$ in  the graph represented in
\autoref{fig:B1SubseteqB3mGraphB1}(a).

In this section we consider the question whether $B_{k}^{m} \subsetneqq B_{k}$ 
holds also for $k \geqslant 2$. 
We  first consider the case $k=2$ in \autoref{sec:B2B2m} and then the remaining 
cases
$k\geqslant 3$ in \autoref{sec:B5B5m}. The case distinction is due to the 
different methods used in  
the investigations. 
\subsection{Relationship between \texorpdfstring{$B_{2}^{m}$}{B2m} and 
\texorpdfstring{$B_{2}$}{B2}}
\label{sec:B2B2m}
The aim of this section is to prove that $B_{2}^{m} \subsetneqq B_{2}$ 
holds. 
For this purpose we show that the graph $H_{1}$ represented in  
\autoref{fig:graph_not_in_b2m} is in $B_{2}$ but not in 
$B_{2}^{m}$.
$H_{1}$ is defined as follows.

\begin{figure}[ht]
    \centering
    \begin{minipage}[b]{0.52\linewidth}
        \centering
        \begin{center}
\begin{tikzpicture}
  [scale=.67]

  \node[above] at (5,10) {$v$};	
  \node[below] at (5,1) {$u$};	  
  \node[vertexSmallBlack] (u) at (5,1) {};
  \node[vertexSmallBlack] (v) at (5,10) {};

  \node[right] at (1,5) {$a_{1}$};	
  \node[right] at (3,5) {$a_{2}$};
  \node[right] at (5,5) {$a_{3}$};	
  \node[left] at (9,5) {$a_{50}$};  	 
  \node[vertexSmallWhite] (a1) at (1,5) {};   
  \node[vertexSmallWhite] (a2) at (3,5) {};       
  \node[vertexSmallWhite] (a3) at (5,5) {};
  \node[vertexSmallWhite] (a50) at (9,5) {};
  
  \node[xshift=0.2cm, yshift=0.2cm] at (2,9) {$b_{1,1}$};	
  \node[xshift=0.2cm, yshift=-0.3cm] at (2,2) {$b_{50,1}$};
  \fill[fill=gray!25] (2,8.25) ellipse (.4 and 0.75);  
  \fill[fill=gray!25] (2,6.75) ellipse (.4 and 0.75);
  \node at (2,8.25) {$H_{2}$};
  \node at (2,6.75) {$H_{2}$};    
  \node[vertexSmallBlack] (b1) at (2,9) {};
  \node[vertexSmallBlack] (b2) at (2,7.5) {};  
  \node[vertexSmallBlack] (b3) at (2,6) {};
  \node[vertexSmallBlack] (b50) at (2,2) {};
  
  \node[xshift=0.0cm, yshift=0.25cm] at (4,9) {$b_{1,2}$};	
  \node[below] at (4,2) {$b_{50,2}$};  
  \fill[fill=gray!25] (4,8.25) ellipse (.4 and 0.75);
  \fill[fill=gray!25] (4,6.75) ellipse (.4 and 0.75); 
  \node at (4,8.25) {$H_{2}$};
  \node at (4,6.75) {$H_{2}$};  
  \node[vertexSmallBlack] (c1) at (4,9) {};
  \node[vertexSmallBlack] (c2) at (4,7.5) {};  
  \node[vertexSmallBlack] (c3) at (4,6) {};
  \node[vertexSmallBlack] (c50) at (4,2) {};

	 \foreach \from/\to in {b3/b50,c3/c50,a3/a50}
  \draw[dotted, shorten <= 1cm, shorten >= 1cm] (\from) -- (\to);
 
  \draw (u) .. controls (0.9,1) .. (a1);
  \draw (u) .. controls (3,1.2) .. (a2);
  \draw (u) .. controls (5,1) .. (a3);
  \draw (u) .. controls (9.1,1) .. (a50);      
  
  \draw (v) .. controls (0.9,10) .. (a1);
  \draw (v) .. controls (3,9.8) .. (a2);
  \draw (v) .. controls (5,10) .. (a3);
  \draw (v) .. controls (9.1,10) .. (a50);
  
  \draw (a1) .. controls (1.2,9) .. (b1);
  \draw (a1) .. controls (1.3,7.5) .. (b2);
  \draw (a1) .. controls (1.4,6) .. (b3);
  \draw (a1) .. controls (1.2,2) .. (b50);
  
  \draw (a2) .. controls (2.8,9) .. (b1);
  \draw (a2) .. controls (2.7,7.5) .. (b2);
  \draw (a2) .. controls (2.6,6) .. (b3);
  \draw (a2) .. controls (2.8,2) .. (b50);         
  
  \draw (a2) .. controls (3.2,9) .. (c1);
  \draw (a2) .. controls (3.3,7.5) .. (c2);
  \draw (a2) .. controls (3.4,6) .. (c3);
  \draw (a2) .. controls (3.2,2) .. (c50);
  
  \draw (a3) .. controls (4.8,9) .. (c1);
  \draw (a3) .. controls (4.7,7.5) .. (c2);
  \draw (a3) .. controls (4.6,6) .. (c3);
  \draw (a3) .. controls (4.8,2) .. (c50);          
  
\draw [decorate,decoration={brace,amplitude=10pt},xshift=0pt,yshift=0pt]
(9,0.6) -- (1,0.6) node [black,midway,yshift=-0.6cm] 
{$50$};  

\draw [decorate,decoration={brace,amplitude=10pt},xshift=0pt,yshift=0pt]
(0.6,2) -- (0.6,9) node [black,midway,xshift=-0.7cm] 
{$50$};

\end{tikzpicture}
\end{center}
        (a)
    \end{minipage}
    \quad
    \begin{minipage}[b]{0.42\linewidth}
        \centering
        \begin{center}
\begin{tikzpicture}
  [scale=.63]
  
  \fill[fill=gray!25] (3,3) ellipse (4 and 2.5);

  \node[vertexBigBlack] (u) at (3,0.5) {};
  \node[vertexBigBlack] (v) at (3,5.5) {};
  \node[above] at (3,5.6) {$b_{i,j}$};
  \node[below] at (u) {$b_{i+1,j}$};

  \node[vertex,draw] (a1) at (2,2) {$c_{1}$};   
  \node[vertex,draw] (a2) at (1,3) {$c_{2}$};       
  \node[vertex,draw] (a3) at (2,4) {$c_{3}$};
  \node[vertex,draw] (a4) at (4,4) {$c_{4}$};   
  \node[vertex,draw] (a5) at (5,3) {$c_{5}$};       
  \node[vertex,draw] (a6) at (4,2) {$c_{6}$};
  
	\foreach \from/\to in {a1/a2,a2/a3,a3/a4,a4/a5,a5/a6,a6/a1,a1/a4,a3/a6,u/a1,u/a6,v/a3,v/a4}
  \draw[] (\from) -- (\to);

  \draw (u) .. controls (1.5,1.5) .. (a2);
  \draw (u) .. controls (1,1) and (-1.5,3) .. (a3);
  \draw (u) .. controls (5,1) and (7.5,3) .. (a4);
  \draw (u) .. controls (4.5,1.5) .. (a5);

  \draw (v) .. controls (1.5,4.5) .. (a2);
  \draw (v) .. controls (1,5) and (-1.5,3) .. (a1);
  \draw (v) .. controls (5,5) and (7.5,3) .. (a6);
  \draw (v) .. controls (4.5,4.5) .. (a5);

\end{tikzpicture}
\end{center}
        (b)
    \end{minipage}
    \caption{(a) The graph $H_{1}$. (b) The graph $H_{2}$ contained in every 
        gray area of $H_{1}$.}
    \label{fig:graph_not_in_b2m}
\end{figure}
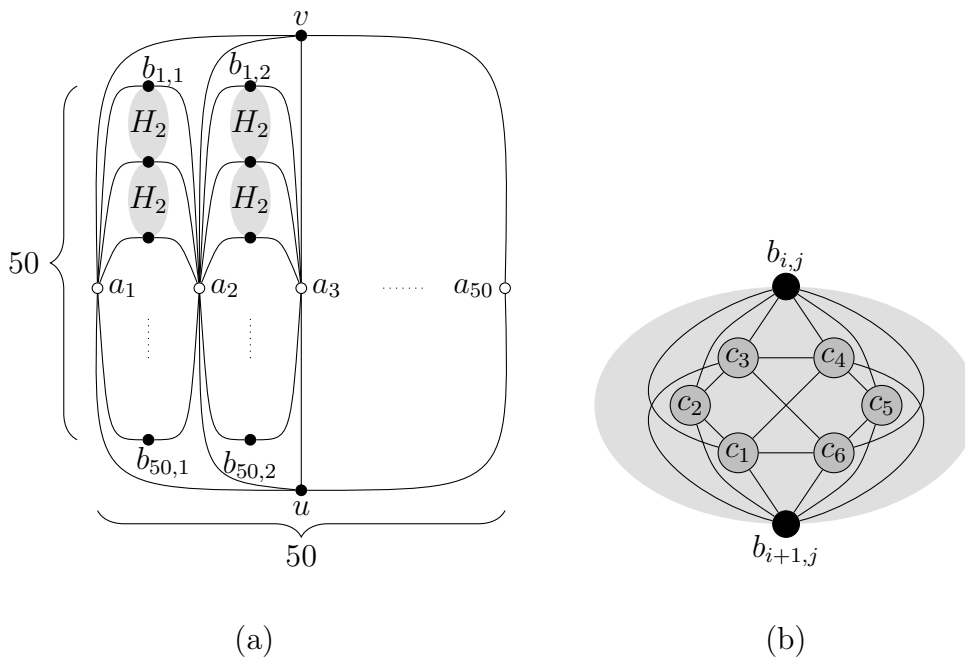

\begin{definition}
The graph $H_{1}$ depicted in \autoref{fig:graph_not_in_b2m} is constructed in 
the following way. The vertices $\{u,v\}$ and $\{a_{1}, \dots,a_{50}\}$ form a 
$K_{2,50}$. Furthermore, for every $1 \leqslant j < 50$ the vertices 
$\{a_{j},a_{j+1}\}$ and $\{b_{1,j}, \dots, b_{50,j}\}$ form a $K_{2,50}$. 
Additional to that for every $1 \leqslant j < 50$ and for every $1 \leqslant i 
< 50$ there is the graph $H_{2}$ of \autoref{fig:graph_not_in_b2m}~(b) placed 
between the vertices $b_{i,j}$ and $b_{i+1,j}$.
\end{definition}

The next result follows from a proof of Heldt, Knauer and Ueckerdt given in~\cite{D}. 
In Proposition 1 of \cite{D}  they  use a similar construction in order to prove that there
is a planar 
graph with treewidth at most $3$ which is not in $B_{2}$. Their construction
builds  also  on the graph $H_1$ (called $G$ in their paper) but the graph
suspended between any two vertices $b_{i,j}$, $b_{i+1,j}$, for $1\leqslant i,j 
< 
50$,
 (called $H$ in their paper) is a $29$-vertex graph different
from $H_2$.  In the first part of the proof of  Proposition~1
Heldt, Knauer and Ueckerdt prove some  properties of $B_2$-EPG representations 
of the subgraph   of $H_1$
 as summarized in the following lemma.  

\begin{lemma}[Heldt, Knauer, Ueckerdt~\cite{D}]
\label{conditionfork250}
In any $B_{2}$-EPG representation of the graph $H_{1}$ depicted in 
\autoref{fig:graph_not_in_b2m} there exist two indices  $i$ and $j$, 
$1\leqslant i,j\leqslant 49$, with the following properties:
\begin{enumerate}[label=(\alph*)]
\setlength{\itemsep}{0pt}
\item the paths 
$P_{b_{i,j}}$ and $P_{b_{i+1,j}}$ 
consist of three segments each,
\item there is a segment $S_{j}$ of the path $P_{a_{j}}$
which completely contains one end segment of 
 $P_{b_{i,j}}$ and one end segment of 
$P_{b_{i+1,j}}$,
\item there is a segment $S_{j+1}$ of the path 
$P_{a_{j+1}}$ which completely contains the other end segments of $P_{b_{i,j}}$ 
and 
$P_{b_{i+1,j}}$,
\item  $S_{j}$ and  $S_{j+1}$ are either both vertical segments or both
horizontal segments. 
\end{enumerate}
\end{lemma}

With this auxiliary result we are able to prove the following lemma.

\begin{lemma}
\label{graph_not_in_b2m}
The graph $H_{1}$ is not in $B_{2}^{m}$.
\end{lemma}
\begin{proof}
Suppose, in order to derive a contradiction,  that $H_{1}$ is in $B_{2}^{m}$. 
Every $B_{2}^{m}$-EPG representation is a $B_{2}$-EPG representation as well, 
therefore \autoref{conditionfork250} holds also for any $B_{2}^{m}$-EPG 
representation of $H_1$. Assume without loss of generality that the center
segment (i.e.\ the second segment) of 
$P_{b_{i,j}}$ is a horizontal segment, that it is above 
the center segment of $P_{b_{i+1,j}}$ and that the 
segment $S_{j}$ of $P_{a_{j}}$ is on the left side of the 
segment $S_{j+1}$ of $P_{a_{j+1}}$. Then the positioning of the 
segments of the paths has to look like in 
\autoref{fig:position_of_theoretical_paths_of_h_a}.

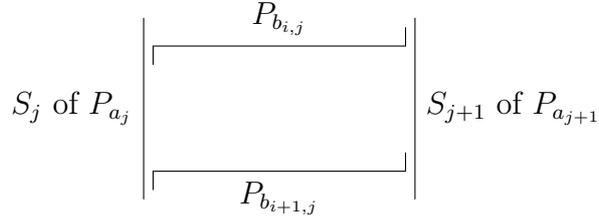
\begin{figure}[ht]
\begin{center}
\begin{tikzpicture}
  [scale=.83]

\newcommand\Vdown{1.0}
\newcommand\Vup{3.0}
\newcommand\Vleft{1.0}
\newcommand\Vright{5.0}
\newcommand\Vspace{0.15}
\newcommand\VlittleIntersection{0.3}
\tikzstyle{rectangleStyle}=[fill=gray!25]

  \draw (\Vleft - \Vspace, \Vdown - 3*\Vspace) -- node[left,pos=0.5] {$S_{j}$ 
  of  $P_{a_{j}}$} (\Vleft-\Vspace,\Vup+3*\Vspace);

  \draw (\Vright + \Vspace, \Vdown - 3*\Vspace) -- node[right,pos=0.5] 
  {$S_{j+1}$ of  $P_{a_{j+1}}$} (\Vright+\Vspace,\Vup+3*\Vspace);

  \draw (\Vleft,\Vdown  - \VlittleIntersection) -- (\Vleft,\Vdown) -- node[below,pos=0.5] {$P_{b_{i+1,j}}$} (\Vright,\Vdown) -- (\Vright,\Vdown +  \VlittleIntersection);
  
  \draw (\Vleft,\Vup  - \VlittleIntersection) -- (\Vleft,\Vup) --  node[above,pos=0.5] {$P_{b_{i,j}}$} (\Vright,\Vup) -- (\Vright,\Vup +  \VlittleIntersection);

\end{tikzpicture}
\end{center}
\caption{A part of the hypothetical $B_{2}^{m}$-EPG representation of $H_{1}$.}
\label{fig:position_of_theoretical_paths_of_h_a}
\end{figure}

Each  vertex $c_{\ell}$, $1\leqslant \ell \leqslant 6$, of the copy of    
$H_{2}$ between $b_{i,j}$ and $b_{i+1,j}$ is adjacent to both 
$b_{i,j}$ and $b_{i+1,j}$, but neither to $a_{j}$ nor to $a_{j+1}$.  Therefore, 
 each of the six paths $P_{c_1}$, $\ldots$, $P_{c_6}$ has  to
 share a grid edge with the center 
 segments of both $P_{b_{i,j}}$ and $P_{b_{i+1,j}}$.
As a result, $P_ {c_{i}}$ starts with   a first horizontal segment  
intersecting 
 the center segment of $P_{b_{i+1,j}}$, continues with a second  vertical 
 segment
 and ends with  a third  horizontal segment  intersecting the center segment of 
 $P_{b_{i,j}}$, for every for $1\leqslant i\leqslant 6$.

Now consider the vertices $c_{1}$, $c_{3}$ and $c_{5}$. They are pairwise 
nonadjacent, so $P_{c_1}$,   $P_{c_3}$, $P_{c_5}$ are
non-intersecting. Therefore, the three vertical segments of these paths are 
disjoint and can be ordered from the left to the right. Let $P_L$, $P_M$ and $P_R$ be 
the path in 
$\{P_{c_{1}}, P_{c_{3}}, P_{c_{5}}\}$ with the left-most, the middle and the 
right-most 
center segment, respectively.
In the following we say that a path $P_{c_i}$ lies to the left of, to the right 
of and on 
 another path $P_{c_j}$ if the center segment of $P_{c_{i}}$ lies to the left 
 of, 
 to the right of  
 and on the center segment of $P_{c_j}$ for some $1 \leqslant i \neq j 
 \leqslant 6$, respectively.

Next take  a closer look at the paths  $P_{c_4}$ and $P_{c_6}$. 
Each of them intersects each of the three paths  $P_{L}$, $P_{M}$ and
$P_{R}$,  since both vertices $c_4$, $c_6$ are 
adjacent to each of $c_{1}$, $c_{3}$ and $c_{5}$. Since $c_4$ and $c_6$ are not adjacent to each
other, $P_{c_4}$ and $P_{c_6}$ do not intersect and hence the vertical segments 
of $P_{c_4}$ and $P_{c_6}$ are disjoint. Assume
without loss of generality 
that $P_{c_4}$ is to the left of $P_{c_6}$.

If $P_{c_4}$ lies  to the right of or on $P_{L}$, 
then $P_{c_6}$ cannot intersect $P_{L}$ on the first or second segment of 
$P_{L}$, because $P_{c_6}$ is to the right of $P_{c_4}$ and does not intersect 
$P_{c_4}$. Therefore, $P_{c_6}$ intersects $P_{L}$ on its third segment. This 
implies that $P_{c_6}$ lies to the left of or on $P_{M}$.
But $P_{c_4}$ is to the left of $P_{c_6}$, which is to the left of or on $P_M$. 
Thus, no point 
of $P_{c_4}$ can lie to the right of the center segment of $P_M$, 
which implies that $P_{c_4}$ does not intersect $P_{R}$, a contradiction.
Analogously, it follows that $P_{c_6}$ cannot lie  to the left of or on $P_R$.

As a result $P_{c_4}$ lies to the left of $P_{L}$ and $P_{c_6}$ lies to the 
right of $P_{R}$. 
$P_{c_4}$ has to intersect $P_{R}$, so the third segment of $P_{L}$ and $P_{M}$ 
are completely contained in the third segment of $P_{c_{4}}$. Similarly 
$P_{c_6}$ has to intersect $P_{L}$, so the first segment of $P_{M}$ and $P_R$ 
are completely contained in the first segment of $P_{c_6}$. 
For an illustration of this configuration see 
\autoref{fig:position_of_theoretical_paths_of_h_b}.

\begin{figure}[ht]
\begin{center}
\begin{tikzpicture}
  [scale=.83]

\newcommand\Vdown{1.0}
\newcommand\Vup{3.0}
\newcommand\Vleft{1.0}
\newcommand\Vright{8.0}
\newcommand\Vspace{0.15}
\newcommand\Voffset{0.5}
\newcommand\VlittleIntersection{0.3}
\newcommand\VbigIntersection{0.8}
\tikzstyle{rectangleStyle}=[fill=gray!25]

  \draw (\Vleft - 3*\Vspace, \Vdown - 3*\Vspace) -- node[left,pos=0.5] {$S_{j}$ 
  of $P_{a_{j}}$} (\Vleft-3*\Vspace,\Vup+3*\Vspace);

  \draw (\Vright + 3*\Vspace, \Vdown - 3*\Vspace) -- node[right,pos=0.5] 
  {$S_{j+1}$ of $P_{a_{j+1}}$} (\Vright+3*\Vspace,\Vup+3*\Vspace);

  \draw (\Vleft - 2*\Vspace,\Vdown  - \VlittleIntersection) -- (\Vleft - 2*\Vspace,\Vdown) -- node[below,pos=0.5] {$P_{b_{i+1,j}}$} (\Vright + 2*\Vspace,\Vdown) -- (\Vright + 2*\Vspace,\Vdown +  \VlittleIntersection);
  
  \draw (\Vleft - 2*\Vspace,\Vup  - \VlittleIntersection) -- (\Vleft - 2*\Vspace,\Vup) --  node[above,pos=0.5] {$P_{b_{i,j}}$} (\Vright + 2*\Vspace,\Vup) -- (\Vright + 2*\Vspace,\Vup +  \VlittleIntersection);

  \draw[dashed] (\Vleft + \VbigIntersection + \Voffset, \Vdown + \Vspace) -- 
  (\Vleft + 2*\VbigIntersection + \Voffset, \Vdown + \Vspace) --  
  node[left,pos=0.5] {$P_L$} (\Vleft + 2*\VbigIntersection + \Voffset, \Vup - 
  \Vspace) -- (\Vleft + 3*\VbigIntersection + \Voffset, \Vup- \Vspace);  

  \draw[dashed] (\Vleft + 2*\VbigIntersection + 2*\Voffset, \Vdown + \Vspace) 
  -- (\Vleft + 3*\VbigIntersection + 2*\Voffset, \Vdown + \Vspace) --  
  node[left,pos=0.5] {$P_M$} (\Vleft + 3*\VbigIntersection + 2*\Voffset, \Vup - 
  \Vspace) -- (\Vleft + 4*\VbigIntersection + 2*\Voffset, \Vup- \Vspace); 
  
  \draw[dashed] (\Vleft + 3*\VbigIntersection + 3*\Voffset, \Vdown + \Vspace) 
  -- (\Vleft + 4*\VbigIntersection + 3*\Voffset, \Vdown + \Vspace) --  
  node[left,pos=0.5] {$P_R$} (\Vleft + 4*\VbigIntersection + 3*\Voffset, \Vup - 
  \Vspace) -- (\Vleft + 5*\VbigIntersection + 3*\Voffset, \Vup- \Vspace);

    \draw[dotted] (\Vleft + 0.5*\Voffset, \Vdown + \Vspace) -- (\Vleft + 
    1*\VbigIntersection , \Vdown + \Vspace) --  node[left,pos=0.5] 
    {$P_{c_{4}}$} (\Vleft + 1*\VbigIntersection, \Vup - 2*\Vspace) -- (\Vleft + 
    4*\VbigIntersection + 4*\Voffset, \Vup- 2*\Vspace); 

    \draw[dotted] (\Vleft + 2.8*\VbigIntersection - \Voffset, \Vdown + 
    2*\Vspace) 
    -- (\Vleft + 6*\VbigIntersection + 3*\Voffset, \Vdown + 2*\Vspace) --  
    node[right,pos=0.5] {$P_{c_{6}}$} (\Vleft + 6*\VbigIntersection + 
    3*\Voffset, \Vup - 
    \Vspace) -- (\Vleft + 6*\VbigIntersection + 4.5*\Voffset, \Vup- \Vspace);

\end{tikzpicture}
\end{center}
\caption{The only possible placement of paths $P_{c_{4}}$, $P_{c_{6}}$ and
$\{P_L, P_M, P_R\} = \{P_{c_{1}}, P_{c_{3}}, P_{c_{5}}\}$
 in the hypothetical 
$B_{2}^{m}$-EPG 
representation of $H_{1}$.}
\label{fig:position_of_theoretical_paths_of_h_b}
\end{figure}
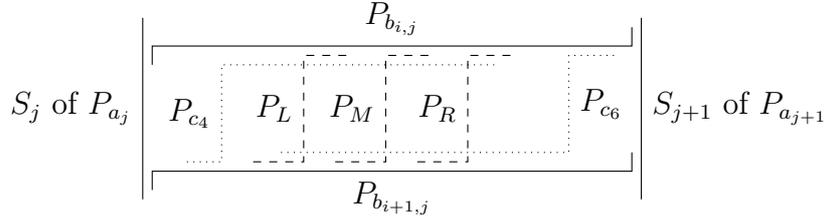

Now consider the path $P_{c_2}$. 
Note that if $P_{c_2}$ intersects the middle segment
of $P_M$, then $P_{c_2}$ is on $P_M$, in which case $P_{c_2}$ also intersects 
the first and third segments of $P_M$. So we can
conclude that if $P_{c_2}$ intersects $P_M$, then $P_{c_2}$ intersects either 
the first segment of $P_M$ or the third segment
of $P_M$. But the former is completely contained in $P_{c_6}$ and the latter 
is completely contained in $P_{c_4}$, which
means that $P_{c_2}$ intersects either $P_{c_4}$ or $P_{c_6}$, which is a 
contradiction to the fact that $c_2$ is nonadjacent to $c_4$
and $c_6$.
We can therefore conclude that $P_{c_2}$ does not intersect $P_M$. 
Since $c_2$ is adjacent to both $c_1$ and $c_3$,
this means that $P_M$ = $P_{c_5}$ and $\{P_L , P_R \} = \{P_{c_1} , P_{c_3}\}$. 
But now it is not possible for $P_{c_2}$ to intersect both
$P_L$ and $P_R$ without intersecting $P_M$.
Hence, $H_{1}$ cannot have a $B_{2}^{m}$-EPG representation.
\end{proof}

After proving that $H_1$  is not in $B_{2}^{m}$, we observe that $H_1$ is in
$B_2$ and obtain  the following theorem.

\begin{theorem}
\label{b2m_noteq_b2}
\label{bkmneqbkfork2}
It holds that $B_{2}^{m} \subsetneqq B_{2}$.
\end{theorem}
\begin{proof}
The fact that $B_{2}^{m} \subseteq B_{2}$ follows by definition. In order to 
see that strict inclusion holds, we consider the graph $H_{1}$ depicted in 
\autoref{fig:graph_not_in_b2m}.

\begin{figure}[ht]
    \centering
    \begin{minipage}[b]{0.9\linewidth}
        \centering
        \begin{center}
\begin{tikzpicture}
  [scale=.7]

\newcommand\V{6.0}
\newcommand\Vsecond{2.5}
\tikzstyle{rectangleStyle}=[fill=gray!25]

  \draw (0.8,0.9) -- (12.2,0.9) node[above,pos=0.8] {$P_{u}$};
  \draw (0.8,\V) -- (12.2,\V) node[above,pos=0.8] {$P_{v}$};
  
  \draw (0.8,1) -- node[below,pos=0.25]  {$P_{a_{1}}$}(1,1) --  (1,\V-0.1) -- (1.2,\V-0.1) ;
    \draw (3.3,1) -- node[below,pos=0.25]  {$P_{a_{2}}$}(3.5,1) --  (3.5,\V-0.1) -- (3.7,\V-0.1) ;
  \draw (5.8,1) -- node[below,pos=0.25]  {$P_{a_{3}}$}(6,1) --  (6,\V-0.1) -- (6.2,\V-0.1) ;
    \draw (8.3,1) -- node[below,pos=0.75]  {$P_{a_{4}}$}(8.5,1) --  (8.5,\V-0.1) -- (8.8,\V-0.1) ;
    \draw (11.8,1) -- node[below,pos=0.25]  {$P_{a_{50}}$}(12,1) --  (12,\V-0.1) -- (12.2,\V-0.1) ;

  \draw[dotted] (9.3,2.5) -- (11.2,2.5);

   \fill[rectangleStyle] (1.2,\V-0.75) rectangle (2.15,\V-0.35);
   \fill[rectangleStyle] (1.2,\V-1.35) rectangle (2.15,\V-0.95);
   \fill[rectangleStyle] (2.35,\V-1.05) rectangle (3.3,\V-0.65);
  \draw (1.1,\V-0.6) -- node[left,pos=0.75]  {$P_{b_{1,1}}$} (1.1,\V-0.4) --   (3.4,\V-0.4) -- (3.4,\V-0.2) ;
  
    \draw (1.1,\V-0.9) --  (1.1,\V-0.7) --   (3.4,\V-0.7)  --  node[right,pos=0.75]  {$P_{b_{2,1}}$}(3.4,\V-0.5) ;

    \draw (1.1,\V-1.2) -- node[left,pos=0.25]  {$P_{b_{3,1}}$} (1.1,\V-1.0) --   (3.4,\V-1.0) -- (3.4,\V-0.8) ;

    \draw (1.1,\V-1.5) --  (1.1,\V-1.3) --   (3.4,\V-1.3) -- node[right,pos=0.25]  {$P_{b_{4,1}}$} (3.4,\V-1.1) ;
    \draw[dotted] (2.25,\V-1.5) -- (2.25,\V-1.9);

    \draw (1.1,\V-2.3) -- node[left,pos=0.25]  {$P_{b_{50,1}}$} (1.1,\V-2.1) --   (3.4,\V-2.1) -- (3.4,\V-1.9) ;

   \fill[rectangleStyle] (6.2,\V-0.75) rectangle (7.15,\V-0.35);
   \fill[rectangleStyle] (6.2,\V-1.35) rectangle (7.15,\V-0.95);
   \fill[rectangleStyle] (7.35,\V-1.05) rectangle (8.3,\V-0.65);
  \draw (6.1,\V-0.6) -- node[left,pos=0.75]  {$P_{b_{1,3}}$} (6.1,\V-0.4) --   (8.4,\V-0.4) -- (8.4,\V-0.2) ;
  
    \draw (6.1,\V-0.9) --  (6.1,\V-0.7) --   (8.4,\V-0.7)  --  node[right,pos=0.75]  {$P_{b_{2,3}}$}(8.4,\V-0.5) ;

    \draw (6.1,\V-1.2) -- node[left,pos=0.25]  {$P_{b_{3,3}}$} (6.1,\V-1.0) --   (8.4,\V-1.0) -- (8.4,\V-0.8) ;

    \draw (6.1,\V-1.5) --  (6.1,\V-1.3) --   (8.4,\V-1.3) -- node[right,pos=0.25]  {$P_{b_{4,3}}$} (8.4,\V-1.1) ;
    \draw[dotted] (7.25,\V-1.5) -- (7.25,\V-1.9);

    \draw (6.1,\V-2.3) -- node[left,pos=0.25]  {$P_{b_{50,3}}$} (6.1,\V-2.1) --   (8.4,\V-2.1) -- (8.4,\V-1.9) ;

   \fill[rectangleStyle] (3.7,\V-\Vsecond-0.75) rectangle (4.65,\V-\Vsecond-0.35);
   \fill[rectangleStyle] (3.7,\V-\Vsecond-1.35) rectangle (4.65,\V-\Vsecond-0.95);
   \fill[rectangleStyle] (4.85,\V-\Vsecond-1.05) rectangle (5.8,\V-\Vsecond-0.65);
  \draw (3.6,\V-\Vsecond-0.6) -- node[left,pos=0.75]  {$P_{b_{1,2}}$} (3.6,\V-\Vsecond-0.4) --   (5.9,\V-\Vsecond-0.4) -- (5.9,\V-\Vsecond-0.2) ;
  
    \draw (3.6,\V-\Vsecond-0.9) --  (3.6,\V-\Vsecond-0.7) --   (5.9,\V-\Vsecond-0.7)  --  node[right,pos=0.75]  {$P_{b_{2,2}}$}(5.9,\V-\Vsecond-0.5) ;

    \draw (3.6,\V-\Vsecond-1.2) -- node[left,pos=0.25]  {$P_{b_{3,2}}$} (3.6,\V-\Vsecond-1.0) --   (5.9,\V-\Vsecond-1.0) -- (5.9,\V-\Vsecond-0.8) ;

    \draw (3.6,\V-\Vsecond-1.5) --  (3.6,\V-\Vsecond-1.3) --   (5.9,\V-\Vsecond-1.3) -- node[right,pos=0.25]  {$P_{b_{4,2}}$} (5.9,\V-\Vsecond-1.1) ;
    \draw[dotted] (4.75,\V-\Vsecond-1.5) -- (4.75,\V-\Vsecond-1.9);

    \draw (3.6,\V-\Vsecond-2.3) -- node[left,pos=0.95]  {$P_{b_{50,2}}$} (3.6,\V-\Vsecond-2.1) --   (5.9,\V-\Vsecond-2.1) -- (5.9,\V-\Vsecond-1.9) ;

\end{tikzpicture}
\end{center}
        (a)
    \end{minipage}
    
    \begin{minipage}[b]{0.9\linewidth}
        \centering
        \begin{center}
\begin{tikzpicture}
  [scale=.7]

\newcommand\Vdown{1.0}
\newcommand\Vup{3.0}
\newcommand\Vspace{0.15}
\tikzstyle{rectangleStyle}=[fill=gray!25]

   \fill[rectangleStyle] (0.3,\Vdown-2*\Vspace) rectangle (8.7,\Vup+2*\Vspace);

  \draw (0.1,\Vdown-\Vspace) -- (8.9,\Vdown-\Vspace) node[below,pos=0.9] {$P_{b_{i+1,j}}$};
  \draw (0.1,\Vup+\Vspace) -- (8.9,\Vup+\Vspace) node[above,pos=0.9] {$P_{b_{i,j}}$};
  
  \draw (0.5,\Vdown) -- (1,\Vdown) --  node[left,pos=0.25]  {$P_{c_{2}}$} (1,\Vup) -- (3,\Vup);
  
  \draw (4.5,\Vdown) -- node[below,pos=0.85]  {$P_{c_{1}}$}(2,\Vdown) --  (2,\Vup-\Vspace) -- (1.5,\Vup-\Vspace);
    
  \draw (6.5,\Vdown) -- node[below,pos=0.25]  {$P_{c_{3}}$}(5,\Vdown) --  (5,\Vup-\Vspace) -- (2.5,\Vup-\Vspace);
  
    \draw (8.5,\Vdown) -- (8,\Vdown) -- node[right,pos=0.25]  {$P_{c_{5}}$} (8,\Vup) -- (6,\Vup);
    
   \draw (3,\Vdown + \Vspace) --(3.5,\Vdown + \Vspace) --  node[right,pos=0.75]  {$P_{c_{4}}$} (3.5,\Vup-2*\Vspace) -- (6.5,\Vup-2*\Vspace);

   \draw (4,\Vdown + \Vspace) -- (7,\Vdown + \Vspace) --  node[left,pos=0.25]  {$P_{c_{6}}$}(7,\Vup-\Vspace) -- (7.5,\Vup-\Vspace);

\end{tikzpicture}
\end{center}
        (b)
    \end{minipage}
    \caption{(a) A $B_{2}$-EPG representation of the graph $H_{1}$ of 
        \autoref{fig:graph_not_in_b2m}. Every gray area 
        represents  the 
        $B_{2}$-EPG representation of $H_2$ depicted in (b).}
    \label{fig:graph_is_in_b2}
\end{figure}

We have already seen in 
\autoref{graph_not_in_b2m} that the graph $H_{1}$ is not in $B_{2}^{m}$. So it 
is enough to show that $H_{1}$ is in $B_{2}$. To this end, consider  a
$B_{2}$-EPG representation of $H_{1}$ given in  \autoref{fig:graph_is_in_b2}.
\end{proof}

Summarizing   $B_{k}^{m} \subsetneqq B_{k}$ holds 
for $k=1$ as shown in~\cite{E} and also  for $k=2$ as shown in this paper.

\subsection{Relationship between \texorpdfstring{$B_{k}^{m}$}{Bkm} and 
\texorpdfstring{$B_{k}$}{Bk} for \texorpdfstring{$k  \in \{3,5\}$ and $k 
\geqslant 7$}{k is 3, k is 5 and k greater equal 7}}
\label{sec:B5B5m}

In this section we use the results from \autoref{sec:KmnBkm} in order to 
investigate the relationship between $B_{k}^{m}$ and $B_{k}$ for $k  \in 
\{3,5\}$ and $k\geqslant 7$.

We start with $k=3$ and prove that $B_{3}^{m} \subsetneqq 
B_{3}$ holds. 
To this end we use a result of~\cite{F} to show that a particular graph
is in $B_{3}$, and then use  results of \autoref{sec:KmnBkm} to prove that this
graph is not in $B_{3}^{m}$.

\begin{lemma}
    \label{b3m_noteq_b3}
    It holds that $B_{3}^{m} \subsetneqq B_{3}$.
    \end{lemma}
    \begin{proof}
      Since  $B^m_3\subseteq B_3$ obviously holds, it is enough to     show that
    $B^m_3\subsetneqq B_3$.
Heldt, Knauer, Ueckerdt~\cite{F}  showed that $b(K_{3,36}) = 3$,  hence $K_{3,36}$ belongs to $B_{3}$.
       Now assume that 
        $K_{3,36}$ is in $B_{3}^{m}$. Then by \autoref{mlbl2} we have
        \begin{align*}
        &&36\left(3 - \left\lceil\frac{4}{2}\right\rceil\right) 
        &\leqslant 
        3\left(2\left\lfloor\frac{4}{2}\right\rfloor\left\lceil\frac{4}{2}\right\rceil
         + 3\right) + \frac{1}{4}3^2,
        \end{align*}
        That is, $36 \le 35.25$, a contradiction. Hence, $K_{3,36}$ is
        not in $B_{3}^{m}$.
                \end{proof}

Now we know that $B_{k}^{m} \subsetneqq B_{k}$ holds for 
$k \leqslant 3$.
Next we show  $B_{5}^{m} \subsetneqq B_{5}$. Similarly as in the case of $k=3$
we  
use a result of~\cite{F} to show that a particular graph
is in $B_{5}$ and then use  results of \autoref{sec:KmnBkm} to prove that 
this graph is not in $B_{5}^{m}$.
 
\begin{lemma}
\label{bkmneqbkfork5}
It holds that $B_{5}^{m} \subsetneqq B_{5}$.
\end{lemma}
\begin{proof}
Since  $B_{5}^{m} \subseteq B_{5}$ obviously holds,  it is enough to show 
$B_{5}^{m} \neq B_{5}$.

Heldt, Knauer, Ueckerdt~\cite{F} showed that $K_{m,n} \in B_{2m-3}$ if 
$n \leqslant m^4 - 2m^3 + \frac{5}{2}m^2 - 2m -4$ (see Theorem~4.5 in \cite{F}). 
For $m = 4$ this implies that $K_{4, 156} \in B_{5}$. Assume that  $K_{4, 156} 
\in B_{5}^{m}$. Then, by \autoref{mlbl} we get
\begin{align*}
&&156(2\cdot 4 - 5 - 2) &\leqslant 5\cdot 3 \cdot 4 + \frac{1}{2}4^2 + 2\cdot 6 
\cdot 4
\end{align*}
 That is, $156\le 116$, a contradiction. So $K_{4,156} \not \in B_{5}^{m}$ but $K_{4,156} \in 
B_{5}$, hence $B_{5}^{m} \neq B_{5}$.
\end{proof}

Finally we show $B_{k}^{m}\subsetneqq B_{k}$ for 
$k \geqslant 7$. To this end, we  use  Lemma~\ref{mlbl}  and
\autoref{KmninBmMinus1}.

\begin{lemma}
\label{bkmneqbkforkgreater7}
 It holds that $B_{k}^{m} \subsetneqq B_{k}$ for $k \geqslant 7$.
\end{lemma}
\begin{proof}
  We first prove the statement for odd $k$. \autoref{KmninBmMinus1} implies
  that $K_{k+1, \frac{1}{4}(k+1)^3 - \frac{1}{2}(k+1)^2 - (k+1) + 4} = 
K_{k+1, \frac{1}{4}k^3 + \frac{1}{4}k^2 - \frac{5}{4}k + \frac{11}{4}} \in 
B_{k}$ for $k\geqslant 3$. Suppose, in order to derive a contradiction,  that
this graph is in $B_{k}^{m}$. 
Then, by  \autoref{mlbl} with $m = k+1$ and $n =  \frac{1}{4}k^3 + \frac{1}{4}k^2 - 
\frac{5}{4}k + \frac{11}{4}$ it follows that
\begin{align*}
&&\left(\frac{1}{4}k^3 + \frac{1}{4}k^2 - \frac{5}{4}k + 
\frac{11}{4}\right)(2(k+1) - k - 2) &\leqslant k^{2}(k+1) + \frac{1}{2}(k+1)^2 
+\\
&&&\quad \quad \quad \quad 2(k+1)^{2}\\
&\Leftrightarrow &\qquad k\left(\frac{1}{4}k^3 + \frac{1}{4}k^2 - \frac{5}{4}k 
+ \frac{11}{4}\right) &\leqslant k^3 + \frac{7}{2}k^2 + 5k + \frac{5}{2}\\
&\Leftrightarrow & \qquad k^4 - 3k^3 - 19k^2 - 9k - 10 &\leqslant 0,
\end{align*}
which is a contradiction for $k \geqslant 7$. Hence, for odd $k \geqslant 7$ 
there is a graph in $B_{k}$ which is not in $B_{k}^{m}$ and therefore  
$B_{k}^{m} \subsetneqq B_{k}$ holds for odd $k \geqslant 7$.

Now  consider the complementary case of  even $k$. \autoref{KmninBmMinus1} 
implies that the graph $K_{k+1, \frac{1}{4}(k+1)^3 - (k+1)^2 + \frac{3}{4}(k+1)} 
= K_{k+1, \frac{1}{4}k^3 - \frac{1}{4}k^2 - \frac{1}{2}k} \in B_{k}$ for 
$k\geqslant 6$.
Suppose, in order to derive a contradiction,  that this graph is in $B_{k}^{m}$. Then, by \autoref{mlbl} with 
$m = k+1$ and $n =  \frac{1}{4}k^3 - \frac{1}{4}k^2 - \frac{1}{2}k$ we get 
\begin{align*}
&& \left(\frac{1}{4}k^3 - \frac{1}{4}k^2 - \frac{1}{2}k\right)(2(k+1) - k - 2) 
&\leqslant k^{2}(k+1) + \frac{1}{2}(k+1)^2 + \\
&&& \quad \quad \quad \quad 2(k+1)^{2}\\
&\Leftrightarrow & \qquad k\left(\frac{1}{4}k^3 - \frac{1}{4}k^2 - 
\frac{1}{2}k\right) &\leqslant k^3 + \frac{7}{2}k^2 + 5k + \frac{5}{2}\\
&\Leftrightarrow &\qquad k^4 - 5k^3 - 16k^2 - 20k - 10 &\leqslant 0,
\end{align*}
which is a contradiction for $k \geqslant 8$. Hence, for even $k \geqslant 8$ there is 
a graph in $B_{k}$ which is not in $B_{k}^{m}$. Therefore  $B_{k}^{m} 
\subsetneqq B_{k}$ for even $k \geqslant 8$ and this completes the proof. 
\end{proof}

\autoref{b3m_noteq_b3}, \autoref{bkmneqbkfork5} and 
\autoref{bkmneqbkforkgreater7} imply the following theorem. 
\begin{theorem}
\label{thm:RelationshipBkmBk}
It holds that $B_{k}^{m} \subsetneqq B_{k}$ for $k=3$, $k=5$ and 
$k \geqslant 7$.
\end{theorem}

Summarizing,   in  \autoref{bkmneqbkfork2} and 
\autoref{thm:RelationshipBkmBk}  we have shown that  $B_{k}^{m} 
\subsetneqq B_{k}$,  for $k \in \{2,3,5\}$ and for  
$k \geqslant 7$,  addressing  herewith  a   question  
raised in~\cite{startpaper}. Recall that $B_1^m\subsetneqq B_1$ was 
 already shown in  \cite{E}. Thus, the only open cases are $k=4$ and $k=6$.  
 We conjecture that  $B_{k}^{m} 
\subsetneqq B_{k}$  holds also
for these two  remaining cases.
\begin{conjecture}
    \label{con:Bk46}
$B_{k}^{m} \subsetneqq B_{k}$ holds also for $k=4$ and $k=6$.
\end{conjecture}
However, this remains an open  question.

\section{Relationship between \texorpdfstring{$B_{k}$}{Bk} and \texorpdfstring{$B_{\ell}^{m}$}{Blm} for \texorpdfstring{$\ell > k$}{l > k}}
\label{sec:BkBlm}

Recall that the inclusions  chains $B_i \subseteq B_{i+1}$  and
    $B^m_i \subseteq    B^m_{i+1}$
 trivially hold for all $i\in \nz$. 
In other words the size of the classes of graphs that have a (monotonic) 
$k$-bend
EPG representation increase with increasing $k$. Also the relationships
$B_0=B_0^m$ and $B_0\subseteq B_1^m$ are trivial. Moreover,
$B_{k}^{m} \subsetneqq B_{k}$ holds for almost all $k \in \nz$,  as shown in
Section~\ref{allBkBkm}. 

This means that in
general the minimum number of bends needed for an EPG representation of a graph
increases when the representing paths on the grid are required to be monotonic.
Analogously, in general  the minimum number of bends needed for an EPG representation of a graph
 decreases as compared to  the minimum number of bends needed in a
monotonic EPG representation. 
Quantifying the magnitude of such an  increase (decrease) arises as a natural question  in
this context. More
generally,  it would be interesting to investigate  the existence of non-trivial functions $f, g\colon \nz\to \nz$
 such that $b^m(G)\leqslant f(b(G))$ and $b(G)\leqslant g(b^m(G))$ holds for 
 all graphs  $G$, or only for all  $G$ belonging to some particular class of graphs.

 To the best of our knowledge questions of this kind   have not  been 
addressed in the literature so far.
In this
section we present some  related results.
In particular, in Section~\ref{sec:BkBlmWith2kMinus9} we show that
the increase (decrease) of the number of bends as
mentioned above  cannot be bounded by one,  in general.
More precisely, by combining the results of  \autoref{KmnInBkNotInB2kMinus89} and
\autoref{thm:KmnInB2mMinus2} with some result known in the literature we show 
that none of the inclusions  $B_{k} \subseteq B_{k+1}^{m}$,   $B_{k+1}^{m}
\subseteq B_{k}$ holds, a result not known so far in the literature.
Then in Section~\ref{sec:B1B3m} we  show that $B_1\subseteq B_3^m$ holds.

\subsection{Relationship between \texorpdfstring{$B_{k}$}{Bk} and 
\texorpdfstring{$B_{2k-9}^{m}$}{B2kMinus9m}}
\label{sec:BkBlmWith2kMinus9}

\begin{theorem}
\label{KmnInBkNotInB2kMinus89}
Let $k \geqslant 5$. If $k$ is odd, then there is a graph which is in $B_{k}$ 
but not in $B_{2k-8}^{m}$. If $k$ is even, there is a graph which is in $B_{k}$ 
but not in $B_{2k-9}^{m}$.
\end{theorem}
\begin{proof}
Consider first the case where  $k$ is odd. In this case 
\autoref{KmninBmMinus1} implies  that $G_k:=K_{k+1, \frac{1}{4}(k+1)^3 - 
\frac{1}{2}(k+1)^2 - (k+1) + 4} = K_{k+1, \frac{1}{4}k^3 + \frac{1}{4}k^2 - 
\frac{5}{4}k + \frac{11}{4}}$ is in $B_{k}$ for $k\geqslant 3$.

Assume that 
$G_k$ belongs to $B_{2k-8}^{m}$ for $k\geqslant 5$. Then, 
\autoref{mlbl} implies 
\begin{align*}
&&\left(\frac{1}{4}k^3 + \frac{1}{4}k^2 - \frac{5}{4}k + \frac{11}{4}\right)8 
&\leqslant (2k-8)(k+1)k +  \frac{1}{2}(k+1)^2 + \\
&&& \quad \quad \quad \quad 2(2k-7)(k+1)\\
&\Leftrightarrow &\qquad 8\left(\frac{1}{4}k^3 + \frac{1}{4}k^2 - \frac{5}{4}k 
+ \frac{11}{4}\right) &\leqslant 2k^3 - \frac{3}{2}k^2 - 17k - \frac{27}{2}\\
&\Leftrightarrow & \qquad  7k^2 + 14k + 71 &\leqslant 0\, 
,
\end{align*}
 which is  a contradiction for $k\geqslant 0$. 
So $G_k$ is not in $B_{2k-8}^{m}$. Hence, for odd $k \geqslant 5$, there is 
a graph in $B_{k}$ which is not in $B_{2k-8}^{m}$.

Consider now the case where $k$ is even.  \autoref{KmninBmMinus1} implies  
$G_k':=K_{k+1, \frac{1}{4}(k+1)^3 - (k+1)^2 + 
\frac{3}{4}(k+1)} = K_{k+1, \frac{1}{4}k^3 - \frac{1}{4}k^2 - \frac{1}{2}k} \in 
B_{k}$ for $k\geqslant 6$. If we assume that $G_k'$  is  in 
$B_{2k-9}^{m}$ for $k\geqslant 6$, we obtain the following inequality  by
applying  \autoref{mlbl}
\begin{align*}
&& \left(\frac{1}{4}k^3 - \frac{1}{4}k^2 - \frac{1}{2}k\right)9 &\leqslant 
(2k-9)(k+1)k +  \frac{1}{2}(k+1)^2 +\\
&&&\quad \quad \quad \quad  2(2k-8)(k+1)\\
&\Leftrightarrow & \qquad 9\left(\frac{1}{4}k^3 - \frac{1}{4}k^2 - 
\frac{1}{2}k\right) &\leqslant 2k^3 - \frac{5}{2}k^2 - 20k - \frac{31}{2}\\
&\Leftrightarrow &\qquad  k^3 + k^2 + 62k + 62 &\leqslant 0\, , 
\end{align*}
 which  is a contradiction for $k \geqslant 0$. Hence,   
 $G_k'$ is in $B_{k}$ but not in 
$B_{2k-9}^{m}$ for even $k \geqslant 6$.
\end{proof}

\autoref{KmnInBkNotInB2kMinus89} reveals that $B_{k} \not 
\subseteq B_{2k-8}^{m}$ for odd $k \geqslant 5$ and that $B_{k} \not \subseteq 
B_{2k-9}^{m}$ for even $k \geqslant 5$.  Thus,   restricting 
the paths of the EPG representation to be monotonic   is  a significant
limitation. \autoref{KmnInBkNotInB2kMinus89}   clearly implies  that $B_{k} \subseteq B_{k+1}^{m}$ does 
 not hold in general.

We can also settle the  
question whether  $B_{k+1}^{m}\subseteq
B_{k}$ holds in general. 
 Indeed, in~\cite{F} it  was proven 
 that  $b(K_{m,n})=2m-2$ for  $m\geqslant 3$ and $n \geqslant
 m^4 - 2m^3 + 5m^2 - 4m + 1$. Hence,  in particular,
 $K_{m, m^4 - 2m^3 + 5m^2 - 4m + 1}$ is in $B_{2m-2}$, but it is not in 
 $B_{2m-3}$.
 On the other hand, \autoref{thm:KmnInB2mMinus2} implies that
 $K_{m, m^4 - 2m^3 + 5m^2 - 4m + 1}$ is in $B_{2m-2}^m$, so
 $B_{2m-2}^m \not \subseteq B_{2m-3}$ for all $m \geqslant 3$.   
Thus,  $B_{k+1}^{m}\subseteq B_{k}$ does not hold in general.

\subsection{Relationship between \texorpdfstring{$B_{1}$}{B1} and \texorpdfstring{$B_{3}^{m}$}{B3m}}
\label{sec:B1B3m}

As mentioned at the beginning of \autoref{sec:BkBlm},  in
general the minimum number of bends needed for an EPG representation of a graph
increases when the paths on the grid are required to be
monotonic. In order  to quantify the amount of this increase  we would like
to find the 
minimum $\ell$ such that $B_{k} \subseteq B_{\ell}^{m}$. 
\autoref{KmnInBkNotInB2kMinus89} shows that $2k-9$ is a lower bound for $\ell$,
i.e.\ 
 $\ell \geqslant 
 2k-9$ for $k \geqslant 5$.
 
 In the following we focus on small values of $k$.  Since  $B_{0}
 = B_{0}^{m}$ holds,   
$1$ is the smallest value of $k$ for which $\ell$ and/or bounds on it are not known.
In the following we show that $B_{1} \subseteq
B_{3}^{m}$, i.e.\  $3$ is an upper bound on  the minimum value of $\ell$ for
which $B_1\subseteq B_\ell^m$.

\begin{theorem}
\label{B1SubseteqB3m}
The inclusion $B_{1} \subseteq B_{3}^{m}$ holds.
\end{theorem}
\begin{proof}
 Let $G$ be
a graph in $B_{1}$. We show that $G$ is in $B_3^m$ by presenting  a monotonic
$B_3$-EPG representation of $G$. The latter is constructed by transforming a
$B_1$-EPG representation of $G$ into a  $B_3^m$-EPG representation of $G$ as
described below.  The transformation is illustrated by means of an example;
Figures~\ref{fig:B1SubseteqB3mGraphB1}(a) 
and~\ref{fig:B1SubseteqB3mGraphB1}(b)
show a graph $G$ and a $B_1$-EPG representation of it, respectively, whereas 
\autoref{fig:B1SubseteqB3mB3m} shows the corresponding 
$B_{3}^m$-EPG representation obtained as a result of the transformation
mentioned above. 

Let 
$R$ be an arbitrary $B_{1}$-EPG representation of $G$. We place another 
copy of the same $B_{1}$-EPG representation to the top right of $R$, 
see~\autoref{fig:B1SubseteqB3mB1Double},  and
then step by step modify both the original $B_1$-EPG representation and its copy   as
described below. At any point in time during this modification process we denote
by 
 $R_1$  and $R_2$ the current  modified  $B_1$-EPG representation and the 
 current
 modified copy of the original  $B_1$-EPG representation, respectively.
 For a vertex $v$ of $G$ we denote by $P_v$, $P_v^1$ and $P_v^2$ the path 
 corresponding to $v$ in $R$, $R_1$ and $R_2$, respectively. 
 At the beginning of the modification process $R_1$ and $R_2$ coincide with the
 original $B_1$-EPG representation and its copy, respectively, as in
 \autoref{fig:B1SubseteqB3mB1Double}.

\begin{figure}[ht]
\centering
\begin{minipage}[b]{0.45\linewidth}
\centering
\begin{center}
\begin{tikzpicture}
  [scale=.7]

  \node[vertex] (c) at (1,1) {$c$};
  \node[vertex] (a) at (3,1) {$a$};
  \node[vertex] (b) at (2,2.7) {$b$};
  \node[vertex] (d) at (4,2.7) {$d$};
  \node[vertex] (e) at (0,2.7) {$e$};
  \node[vertex] (f) at (2,-0.7) {$f$};
  \node[vertex] (g) at (5.7,3.8) {$g$};

  \foreach \from/\to in {a/b,a/c,a/d,a/f,a/g,b/c,b/e,b/d,c/e,c/f,g/e,g/d}
  \draw (\from) -- (\to);

\end{tikzpicture}
\end{center}
(a)
\end{minipage}
\quad
\begin{minipage}[b]{0.45\linewidth}
\centering
\begin{center}
\begin{tikzpicture}
  [scale=.7]

\def \x {0}
\def \y {0}

  \draw (1+\x,3+\y) -- (6+\x,3+\y) -- (6+\x,4+\y) node[left,pos=0.5] {$P_{a}$};

  \draw (3+\x,5+\y) -- (3+\x,3.15+\y) -- (5+\x,3.15+\y) node[above,pos=0.25] {$P_{b}$};

  \draw (4+\x,2.85+\y) -- (6+\x,2.85+\y) -- (6+\x,2+\y) node[left,pos=0.5] {$P_{d}$};

  \draw (1+\x,3.15+\y) -- (2.85+\x,3.15+\y) -- (2.85+\x,5+\y) node[left,pos=0.2] {$P_{c}$};

  \draw (1+\x,2.85+\y) -- (2+\x,2.85+\y) node [below,pos=0.5] {$P_{f}$};
  
  \draw (6.15+\x,2+\y) -- (6.15+\x,4.85+\y) node [right,pos=0.5] {$P_{g}$} -- (5+\x,4.85);

  \draw (3.15+\x,4+\y) -- (3.15+\x,5+\y) node [right,pos=0.5] {$P_{e}$} -- (6,5);

\end{tikzpicture}
\end{center}
(b)
\end{minipage}
\caption{(a) A graph $G$. (b) The $B_{1}$-EPG representation $R$ of $G$.}
\label{fig:B1SubseteqB3mGraphB1}
\end{figure}
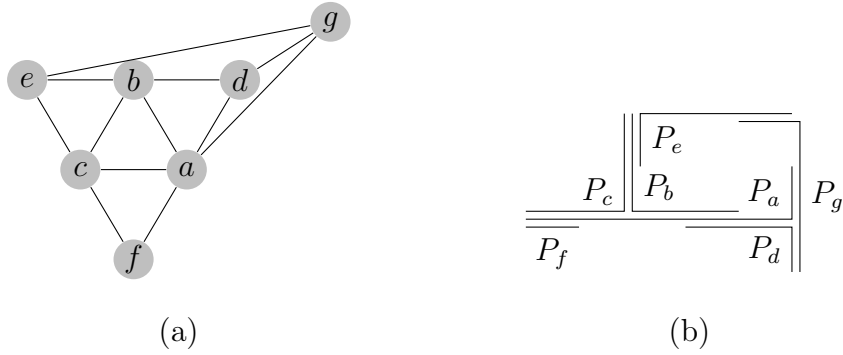

\begin{figure}[ht]
\begin{center}
\begin{tikzpicture}
  [scale=.7]

\draw[step=1cm,gray!25,line width=5pt] (.7,1.7) grid (13.3,9.3);

\foreach \x/\y\n in {0/0/1,7/4/2}
{

  \draw (1+\x,3+\y) -- (6+\x,3+\y) -- (6+\x,4+\y) node[left,pos=0.5] 
  {$P_{a}^\n$};

  \draw (3+\x,5+\y) -- (3+\x,3.15+\y) -- (5+\x,3.15+\y) node[above,pos=0.25] 
  {$P_{b}^\n$};

  \draw (4+\x,2.85+\y) -- (6+\x,2.85+\y) -- (6+\x,2+\y) node[left,pos=0.5] 
  {$P_{d}^\n$};

  \draw (1+\x,3.15+\y) -- (2.85+\x,3.15+\y) -- (2.85+\x,5+\y) 
  node[left,pos=0.2] {$P_{c}^\n$};

  \draw (1+\x,2.85+\y) -- (2+\x,2.85+\y) node [below,pos=0.5] {$P_{f}^\n$};
  
  \draw (6.15+\x,2+\y) -- (6.15+\x,4.85+\y) node [right,pos=0.5] {$P_{g}^\n$} 
  -- (5+\x,4.85+\y);

  \draw (3.15+\x,4+\y) -- (3.15+\x,5+\y) node [right,pos=0.5] {$P_{e}^\n$} -- 
  (6+\x,5+\y);

}

\end{tikzpicture}
\end{center}
\caption{The grid with the two copies $R_1$ and $R_2$ of $R$ shown in
  \autoref{fig:B1SubseteqB3mGraphB1}(b).}
\label{fig:B1SubseteqB3mB1Double}
\end{figure}
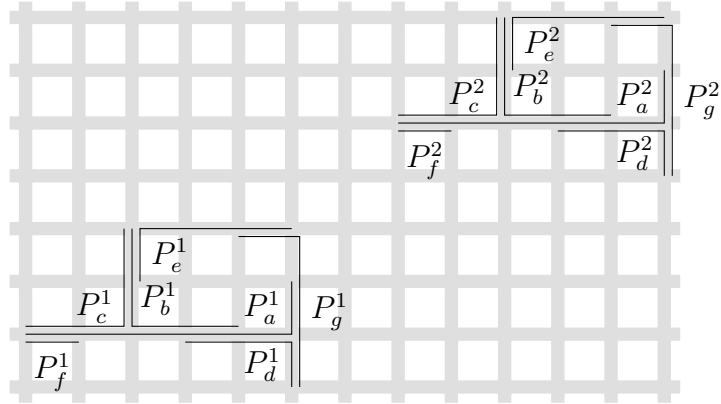

Now consider the vertices  of $G$ one by one in an  arbitrary order and for 
every vertex perform the modifications  described below.
Let $v$ be the currently considered vertex.   The modification of $R_1$ is driven by the
horizontal segment of  the path $P_v$, if any,  whereas the modification of $R_2$ is driven by the
vertical segment of $P_v$, if any.  
If $P_{v}$ has a horizontal segment, we modify $R_1$ as follows. We introduce a new   vertical grid 
line  $L^{|}_{v}$ directly  to the  left of  the vertical grid 
line containing the right end point of the horizontal segment of  $P_{v}^1$ in $R_1$ 
and shorten the horizontal segment of $P_v^1$ to end in $L^{|}_{v}$ instead of  
ending at the original right end point.
Then,  if  the path $P_v$ contains a vertical segment which 
starts at
the original right end point of the horizontal segment mentioned above, we 
modify   
$P_v^1$  in $R_1$ by shifting its vertical segment to lie on  $L^{|}_{v}$. 

If $P_{v}$ has a vertical segment, we modify $R_2$ as follows.
We introduce a new   horizontal grid  line  $L^{-}_{v}$ directly  beneath   the horizontal grid 
line containing the lower end point of the vertical  segment of 
$P_{v}^2$ in $R_2$ and extend the vertical segment of $P_v^2$ until $L^{-}_{v}$.
Then,  if  the path $P_v$ contains a horizontal  segment 
which starts at
the original lower end point of the vertical  segment mentioned above, we 
modify   
$P_v^2$  in $R_2$ by shifting its horizontal segment to lie on  $L^{-}_{v}$. 
An example of the modified grid and paths and the final $R_1$, $R_2$
for the graph in  
\autoref{fig:B1SubseteqB3mGraphB1}~(a)   can be seen in 
\autoref{fig:B1SubseteqB3mNewGridlines}.

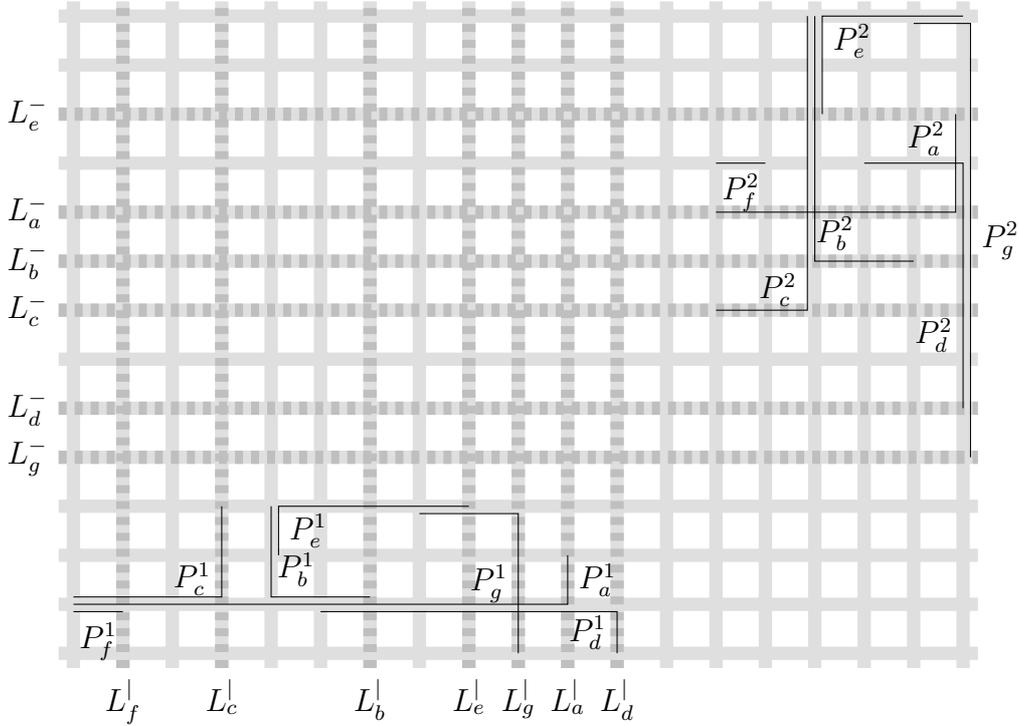
\begin{figure}[ht]
\begin{center}
\begin{tikzpicture}
  [scale=.65]

\def \gridUp {15.3}
\def \gridDown {1.7}
\def \gridLeft {0.7}
\def \gridRight {19.3}

\draw[step=1cm,gray!25,line width=5pt] (\gridLeft,\gridDown) grid (\gridRight,\gridUp);

\def \x {0}
\def \y {0}
  
  \draw[gray!50,dashed,line width = 5pt] (2,\gridDown) -- (2,\gridUp);  
  
      \draw[gray!50,dashed,line width = 5pt] (4,\gridDown) -- (4,\gridUp);  
    
  \draw[gray!50,dashed,line width = 5pt] (7,\gridDown) -- (7,\gridUp);  
          
  \draw[gray!50,dashed,line width = 5pt] (9,\gridDown) -- (9,\gridUp);  
  
      \draw[gray!50,dashed,line width = 5pt] (10,\gridDown) -- (10,\gridUp); 

    \draw[gray!50,dashed,line width = 5pt] (11,\gridDown) -- (11,\gridUp); 
    
      \draw[gray!50,dashed,line width = 5pt] (12,\gridDown) -- (12,\gridUp);

\def \x {13}
\def \y {6}

   \draw[gray!50,dashed,line width = 5pt] (\gridLeft,\y) -- (\gridRight,\y);   
   
   \draw[gray!50,dashed,line width = 5pt] (\gridLeft,\y+1) -- (\gridRight,\y+1); 
   
   \draw[gray!50,dashed,line width = 5pt] (\gridLeft,\y+3) -- (\gridRight,\y+3);
   
  \draw[gray!50,dashed,line width = 5pt] (\gridLeft,\y+4) -- (\gridRight,\y+4); 
  
  \draw[gray!50,dashed,line width = 5pt] (\gridLeft,\y+5) -- (\gridRight,\y+5);      
                   
   \draw[gray!50,dashed,line width = 5pt] (\gridLeft,\y+7) -- (\gridRight,\y+7);

\def \x {0}
\def \y {0}
  
  \draw (1+\x,2.85+\y) -- (2+\x,2.85+\y) node [below,pos=0.5] {$P_{f}^1$};
  
    \draw (1+\x,3.15+\y) -- (3+\x+1,3.15+\y) -- (3+\x+1,5+\y) 
    node[left,pos=0.2] {$P_{c}^1$};
    
  \draw (3+\x+2,5+\y) -- (3+\x+2,3.15+\y) -- (5+\x+2,3.15+\y) 
  node[above,pos=0.25] {$P_{b}^1$};
  
        \draw (5.15+\x,4+\y) -- (5.15+\x,5+\y) node [right,pos=0.5] {$P_{e}^1$} 
        -- (9+\x,5+\y);
        
  \draw (10+\x,2+\y) -- (10+\x,4.85+\y) node [left,pos=0.5] {$P_{g}^1$} -- 
  (8+\x,4.85+\y);        

  \draw (1+\x,3+\y) -- (11+\x,3+\y) -- (11+\x,4+\y) node[right,pos=0.5] 
  {$P_{a}^1$};  
  
  \draw (6+\x,2.85+\y) -- (12+\x,2.85+\y) -- (12+\x,2+\y) node[left,pos=0.5] 
  {$P_{d}^1$};

\def \x {13}
\def \y {6}

  \draw (4+\x,3+\y+3) -- (6+\x,3+\y+3) -- (6+\x,1+\y) node[left,pos=0.7] 
  {$P_{d}^2$};
  
  \draw (6.15+\x,\y) -- (6.15+\x,4.85+\y+4) node [right,pos=0.5] {$P_{g}^2$} 
  -- (5+\x,4.85+\y+4);      
  
    \draw (1+\x,3+\y) -- (2.85+\x,3+\y) -- (2.85+\x,5+\y +4) 
    node[left,pos=0.07] {$P_{c}^2$};
    
  \draw (3+\x,5+\y+4) -- (3+\x,3+\y+1) -- (5+\x,3+\y+1) node[above,pos=0.2] 
  {$P_{b}^2$};
  
  \draw (1+\x,3+\y+2) -- (5.85+\x,3+\y+2) -- (5.85+\x,4+\y+3) 
  node[left,pos=0.75] {$P_{a}^2$};   

    \draw (1+\x,3+\y+3) -- (2+\x,3+\y+3) node [below,pos=0.5] {$P_{f}^2$};
    
  \draw (3.15+\x,4+\y+3) -- (3.15+\x,5+\y+4) node [right,pos=0.75] {$P_{e}^2$} 
  -- (6+\x,5+\y+4);

	\node[below] at (2,\gridDown) {$L^{|}_{f}$};
	\node[below] at (4,\gridDown) {$L^{|}_{c}$};   
	\node[below] at (7,\gridDown) {$L^{|}_{b}$};   
	\node[below] at (9,\gridDown) {$L^{|}_{e}$};   
	\node[below] at (10,\gridDown) {$L^{|}_{g}$};   
	\node[below] at (11,\gridDown) {$L^{|}_{a}$};   
	\node[below] at (12,\gridDown) {$L^{|}_{d}$};   				    
	\node[left] at (\gridLeft,7) {$L^{-}_{d}$};       
	\node[left] at (\gridLeft,6) {$L^{-}_{g}$};  
	\node[left] at (\gridLeft,9) {$L^{-}_{c}$};  
	\node[left] at (\gridLeft,10) {$L^{-}_{b}$};  
	\node[left] at (\gridLeft,11) {$L^{-}_{a}$};  
	\node[left] at (\gridLeft,13) {$L^{-}_{e}$};  				
  
\end{tikzpicture}
\end{center}
\caption{The final status of the modifications $R_1$ and $R_2$ of  the
  $B_1$-EPG representation $R$  shown in
  \autoref{fig:B1SubseteqB3mGraphB1}(b) and its copy. This final status is
  obtained for any order $\prec$ of vertices in which $c\prec b$  and $e \prec
  g\prec a\prec d$.}
\label{fig:B1SubseteqB3mNewGridlines}
\end{figure}

Now we construct a  $B_{3}^{m}$-EPG representation of $G$ with a path $Q_{v}$ for 
every vertex $v$ in the following way. If the path $P_{v}$ consists of  a
single  horizontal  segment, we define $Q_{v}$ as the horizontal segment of $P_{v}^1$ in $R_{1}$ 
and call this segment the \emph{lower} segment of $Q_v$.
If the path $P_{v}$ consists of a single vertical segment, we define $Q_{v}$ as
the vertical segment of $P_{v}^2$ in $R_{2}$ and call this segment the \emph{upper} segment  of $Q_v$. 
If the path $P_{v}$ contains  a horizontal and a  vertical segment,
then the path $Q_{v}$ starts with   the horizontal segment of  $P_{v}^1$  in $R_{1}$;
this segment is called  the \emph{lower} segment of $Q_v$. Further the path 
$Q_{v}$  continues with a vertical segment lying on  the vertical grid line
$L^{|}_{v}$ and ending at the intersection of  $L^{|}_{v}$ and $L^{-}_{v}$.
This intersection is the upper end point of  this segment.
Starting at  this grid point  $Q_v$  proceeds  with a  horizontal segment lying on $L^{-}_{v}$ 
 until it reaches the vertical grid line containing the vertical segment  of
 $P_{v}^2$ in $R_{2}$. Finally $Q_v$  ends  with the vertical segment of 
$P_{v}^2$ in $R_{2}$;  this segment is called  the \emph{upper} segment of $Q_v$.
The result of this construction for  the graph given in \autoref{fig:B1SubseteqB3mGraphB1}(a)
and its $B_1$-EPG representation $R$ is depicted in \autoref{fig:B1SubseteqB3mB3m}.

Observe that this construction has  the following properties.  
If $P_v$ contains two segments, then $Q_v$ contains $4$ segments, the lower one
being the horizontal segment of $P_v^1$  in $R_1$ and the upper one being  the vertical segment of $P_v^2$ in $R_2$.
The two remaining  segments, a vertical and a horizontal one,  are contained in 
the two additionally introduced  grid lines that are used by no other path, 
because every path $Q_v$ uses only the additional grid lines $L^{|}_{v}$ and 
$L^{-}_{v}$  introduced exclusively for the vertex $v$.
If $P_v$ consists of one  horizontal (vertical) segment, then $Q_v$ consists also of  one  horizontal
(vertical) segment which coincides with the corresponding segment of $P_v^1$ ($P_v^2$) in $R_1$ ($R_2$) 
and is a lower (upper) segment.
It is easy to see that every path $Q_v$ in this construction  is  monotonic and bends at most $3$ times.

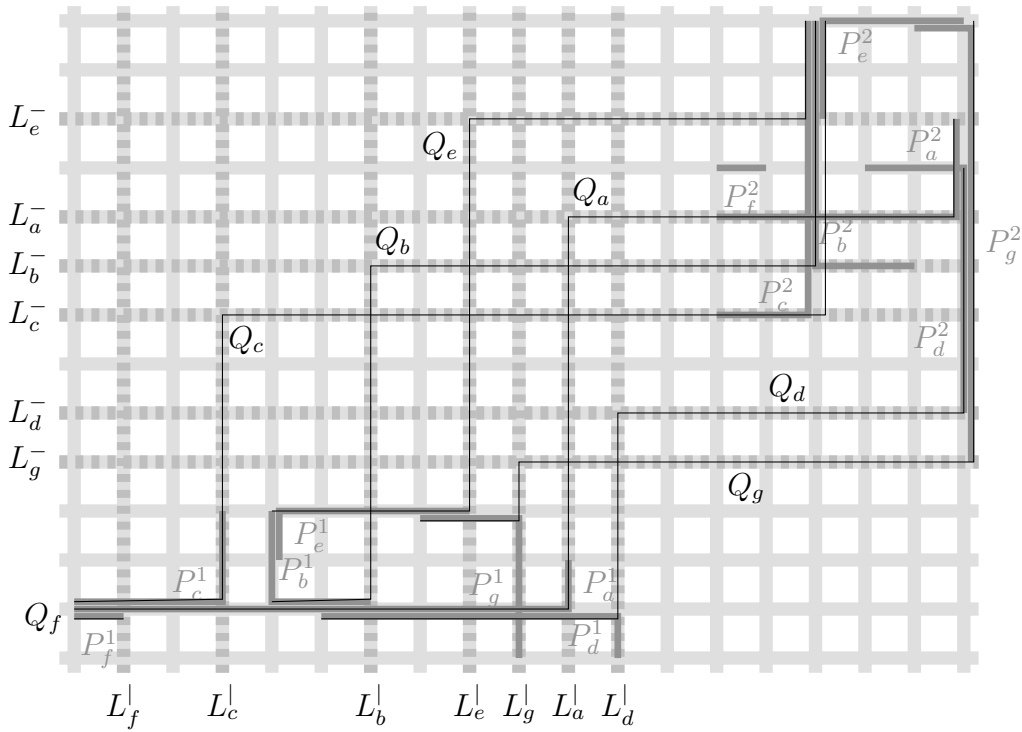
\begin{figure}[ht]
\begin{center}
\begin{tikzpicture}
  [scale=.65]
  
\tikzstyle{Q} = []
\tikzstyle{P} = [gray!85, line width = 2.5pt] 

\def \gridUp {15.3}
\def \gridDown {1.7}
\def \gridLeft {0.7}
\def \gridRight {19.3}

\draw[step=1cm,gray!25,line width=5pt] (\gridLeft,\gridDown) grid (\gridRight,\gridUp);

\def \x {0}
\def \y {0}

  \draw[gray!50,dashed,line width = 5pt] (2,\gridDown) -- (2,\gridUp);  
  
      \draw[gray!50,dashed,line width = 5pt] (4,\gridDown) -- (4,\gridUp);  
    
  \draw[gray!50,dashed,line width = 5pt] (7,\gridDown) -- (7,\gridUp);  
          
  \draw[gray!50,dashed,line width = 5pt] (9,\gridDown) -- (9,\gridUp);  
  
      \draw[gray!50,dashed,line width = 5pt] (10,\gridDown) -- (10,\gridUp); 

    \draw[gray!50,dashed,line width = 5pt] (11,\gridDown) -- (11,\gridUp); 
    
      \draw[gray!50,dashed,line width = 5pt] (12,\gridDown) -- (12,\gridUp);

\def \x {13}
\def \y {6}

   \draw[gray!50,dashed,line width = 5pt] (\gridLeft,\y) -- (\gridRight,\y);   
   
   \draw[gray!50,dashed,line width = 5pt] (\gridLeft,\y+1) -- (\gridRight,\y+1); 
   
   \draw[gray!50,dashed,line width = 5pt] (\gridLeft,\y+3) -- (\gridRight,\y+3);
   
  \draw[gray!50,dashed,line width = 5pt] (\gridLeft,\y+4) -- (\gridRight,\y+4); 
  
  \draw[gray!50,dashed,line width = 5pt] (\gridLeft,\y+5) -- (\gridRight,\y+5);      
                   
   \draw[gray!50,dashed,line width = 5pt] (\gridLeft,\y+7) -- (\gridRight,\y+7);

\def \x {0}
\def \y {0}
  
  \draw[P] (1+\x,2.85+\y) -- (2+\x,2.85+\y) node [below,pos=0.5] {$P_{f}^1$};
  
    \draw[P]  (1+\x,3.15+\y) -- (3+\x+1,3.15+\y) -- (3+\x+1,5+\y) 
    node[left,pos=0.2] {$P_{c}^1$};
    
  \draw[P]  (3+\x+2,5+\y) -- (3+\x+2,3.15+\y) -- (5+\x+2,3.15+\y) 
  node[above,pos=0.25] {$P_{b}^1$};
  
        \draw[P]  (5.15+\x,4+\y) -- (5.15+\x,5+\y) node [right,pos=0.5] 
        {$P_{e}^1$} -- (9+\x,5+\y);
        
  \draw[P]  (10+\x,2+\y) -- (10+\x,4.85+\y) node [left,pos=0.5] {$P_{g}^1$} -- 
  (8+\x,4.85+\y);        

  \draw[P]  (1+\x,3+\y) -- (11+\x,3+\y) -- (11+\x,4+\y) node[right,pos=0.5] 
  {$P_{a}^1$};  
  
  \draw[P]  (6+\x,2.85+\y) -- (12+\x,2.85+\y) -- (12+\x,2+\y) 
  node[left,pos=0.5] {$P_{d}^1$};

\def \x {13}
\def \y {6}

  \draw[P]  (4+\x,3+\y+3) -- (6+\x,3+\y+3) -- (6+\x,1+\y) node[left,pos=0.7] 
  {$P_{d}^2$};
  
  \draw[P]  (6.15+\x,\y) -- (6.15+\x,4.85+\y+4) node [right,pos=0.5] 
  {$P_{g}^2$} -- (5+\x,4.85+\y+4);      
  
    \draw[P]  (1+\x,3+\y) -- (2.85+\x,3+\y) -- (2.85+\x,5+\y +4) 
    node[left,pos=0.07] {$P_{c}^2$};
    
  \draw[P]  (3+\x,5+\y+4) -- (3+\x,3+\y+1) -- (5+\x,3+\y+1) node[above,pos=0.2] 
  {$P_{b}^2$};
  
  \draw[P]  (1+\x,3+\y+2) -- (5.85+\x,3+\y+2) -- (5.85+\x,4+\y+3) 
  node[left,pos=0.75] {$P_{a}^2$};   

    \draw[P]  (1+\x,3+\y+3) -- (2+\x,3+\y+3) node [below,pos=0.5] {$P_{f}^2$};
    
  \draw[P]  (3.15+\x,4+\y+3) -- (3.15+\x,5+\y+4) node [right,pos=0.75] 
  {$P_{e}^2$} -- (6+\x,5+\y+4);

  \draw[Q] (1,2.8) -- (2,2.8) node [left,pos=0.01] {$Q_{f}$};  
  
      \draw[Q] (1,3.15) -- (4,3.2) -- (4,9) -- (16.2,9) node[below,pos=0.04] {$Q_{c}$} -- (16.2,15);
      
  \draw[Q] (5,3.15) -- (7,3.2) -- (7,10)  -- (16,10) node[above,pos=0.05] {$Q_{b}$} -- (16,15);
  
   \draw[Q] (5,5) -- (9,5) --  node [left,pos=0.93] {$Q_{e}$} (9,13) -- (15.8,13) -- (15.8,15);
        
  \draw[Q] (8,4.8) -- (10,4.8) -- (10,6)  -- (19.2,6) node [below,pos=0.5] 
  {$Q_{g}$} -- (19.2,15);        

  \draw[Q] (1,3) -- (11,3) -- (11,11)  -- (18.8,11) node[above,pos=0.06] {$Q_{a}$}-- (18.8,13);  
  
  \draw[Q] (6,2.8) -- (12,2.8) -- (12,7)  -- (19,7) node[above,pos=0.49] 
  {$Q_{d}$} -- (19,12);

	\node[below] at (2,\gridDown) {$L^{|}_{f}$};
	\node[below] at (4,\gridDown) {$L^{|}_{c}$};   
	\node[below] at (7,\gridDown) {$L^{|}_{b}$};   
	\node[below] at (9,\gridDown) {$L^{|}_{e}$};   
	\node[below] at (10,\gridDown) {$L^{|}_{g}$};   
	\node[below] at (11,\gridDown) {$L^{|}_{a}$};   
	\node[below] at (12,\gridDown) {$L^{|}_{d}$};   				    
	\node[left] at (\gridLeft,7) {$L^{-}_{d}$};       
	\node[left] at (\gridLeft,6) {$L^{-}_{g}$};  
	\node[left] at (\gridLeft,9) {$L^{-}_{c}$};  
	\node[left] at (\gridLeft,10) {$L^{-}_{b}$};  
	\node[left] at (\gridLeft,11) {$L^{-}_{a}$};  
	\node[left] at (\gridLeft,13) {$L^{-}_{e}$};

\end{tikzpicture}
\end{center}
\caption{The obtained $B_{3}^{m}$-EPG representation of the graph given in 
\autoref{fig:B1SubseteqB3mGraphB1}(a).}
\label{fig:B1SubseteqB3mB3m}
\end{figure}

 What is left to  show is that the above  construction indeed leads to an EPG representation of $G$,
 i.e.\ that any two paths $Q_{v}$ and $Q_{v'}$  intersect if and 
only if the vertices $v$ and $v'$ are adjacent in $G$. To this end,  it is 
enough to show that two paths $Q_{v}$ and $Q_{v'}$ intersect, if and 
only if the paths $P_{v}$ and $P_{v'}$ intersect in the original  $B_{1}$-EPG representation $R$.

Assume $Q_{v}$ and $Q_{v'}$ intersect.
First consider the case that at least one of $Q_{v}$ and $Q_{v'}$ consists of 
only one segment. Assume without loss of generality that  $Q_{v}$
consists of one  horizontal segment.
Due to the properties of the construction this segment of $Q_{v}$ is a lower 
segment and hence the unique segment of $P_{v}^1$ in $R_1$. Consequently,  
again due to  the properties of the construction,  the segment of   $Q_{v'}$
intersecting $Q_{v}$ is the horizontal segment of $P_{v'}^1$ in $R_1$. Hence $P_{v}^1$ and $P_{v'}^1$   
intersect in the final $R_1$ on their horizontal segments. By construction this 
is only the case if $P_{v}$ and $P_{v'}$ intersect on their horizontal segments 
in $R$, because during the update of $R_1$ only vertical segments of paths are 
moved into  new grid lines in such a way that no new intersections are created.

Now assume that both  paths $Q_{v}$ and $Q_{v'}$ consist of more than one segment. 
There are no intersections of the paths in any additionally  introduced grid 
lines  because every additionally  introduced  grid line  is related to one  vertex
and the additionally  introduced  grid line related to different vertices are different. 
Moreover, by construction every additionally introduced vertical grid line
contains at most one segment of the path $P_v^1$ in $R_1$ representing the vertex $v$ to
which the line is related. Analogously every additionally introduced
horizontal  grid line contains at most one segment of the path $P_v^2$ in $R_2$
representing the vertex $v$ to which the line is related.  These considerations together with the fact that
$R_{1}$ and $R_{2}$ do not share any grid lines imply that  the intersection of  
$Q_{v}$ and $Q_{v'}$ involves either the lower segments of each path, 
or it involves the  upper segments of each path. 
Consequently,  according to the properties of the construction,  the paths   $Q_{v}$ and $Q_{v'}$ intersect in
their lower segments (in $R_1$)  or in their upper segments (in $R_2$).
In both situations we can proceed as in the previous case.

Next we show  the other direction of the equivalence, that is  we assume  that  
$P_{v}$ and $P_{v'}$ intersect in the original $B_1$-EPG representation $R$ of 
$G$ and show that  also $Q_{v}$ and $Q_{v'}$  intersect.
By construction,   if  $P_{v}$ and $P_{v'}$ intersect in a horizontal
grid line, then  the modified paths $P_{v}^1$ and $P_{v'}^1$ 
intersect in a  horizontal grid line  in $R_{1}$ at all times. Thus, the
properties of the construction imply the intersection of the lower   segments of $Q_{v}$ and $Q_{v'}$. 
Analogously,   if  $P_{v}$ and $P_{v'}$ intersect in a  vertical  grid line, then
the modified paths $P_{v}^2$ and  $P_{v'}^2$  intersect in a vertical grid line in $R_{2}$ at all times,
and the properties of the construction imply the intersection of the upper segments of $Q_{v}$ and $Q_{v'}$.
\end{proof}

Notice that it is an  open question  whether the result of
\autoref{B1SubseteqB3m}  is  the  best possible, that is  whether 
$\ell = 3$ is really the minimum $\ell$ such that $B_{1} \subseteq 
B_{\ell}^{m}$ or whether even $B_{1} \subseteq B_{2}^{m}$ holds.

We conclude this section with a few comments related to the size of the grid in 
EPG representations, that is  the number of horizontal and vertical grid lines
used by the paths in the EPG representation. Recently this question was
investigated by Biedl, Derka, Dujmovi\'c and Morin~\cite{TSmallGrid}.
The  size of the  $B_{3}^{m}$-EPG representation obtained by the construction in   the proof
of  \autoref{B1SubseteqB3m} depends on   the size of the $B_{1}$-EPG representation of the graph;
in the worst  case the constructed $B_{3}^{m}$-EPG representation     uses  twice
as many horizontal grid lines and twice as many vertical grid line as compared
to the original  $B_{1}$-EPG representation and an additional horizontal and vertical grid line for every vertex.
This gives rise to the natural  question whether the construction given in the proof of \autoref{B1SubseteqB3m}
is the best possible with respect to the size of the grid. Currently we cannot  answer  this question. 

When considering  the dependency of the grid size of the $B_{3}^{m}$-EPG representation on the
grid size of the starting  $B_{1}$-EPG representation in the  construction
given in  \cite{TSmallGrid}, another natural question arises.
What is the smallest   possible size of the grid in a  $B_1$-EPG representation
of a $B_1$-EPG graph?
In the small  EPG representations dealt with in \cite{TSmallGrid} no 
fixed number of bends is considered,   so the question above  is also open.

\section{Conclusions and Open Problems}
\label{sec:Conclusions}

In this paper, we investigated the relationship between the classes 
$B_k$ and $B_\ell^m$ for different values of $k, \ell\in \nz$.

In particular, we considered the  bend number and  the
monotonic bend number of complete bipartite graphs. We extended the already known  result
$b(K_{m,n})\leqslant 2m-2$  (see \cite{F}) to the monotonic bend number, that is  we proved
$b^m(K_{m,n})\leqslant 2m-2$  for any $3 \leqslant m \leqslant n$, and showed  
that the  upper bound $2m-2$ is  attained  for smaller values of $n$ in the monotonic case.

As  auxiliary results we derived two different inequalities which  hold whenever 
$K_{m,n}$ is in $B_{k}^{m}$. We used these inequalities to prove the strict
inclusion  $B_{k}^{m} \subsetneqq B_{k}$  for $k \in \{3,5\}$ and $k \geqslant 7$. 
Furthermore, we showed that $B_{2}^{m} \subsetneqq B_{2}$ by specifying a 
particular graph which is in $B_{2}$ but not in $B_{2}^{m}$. 
Thus, we gave an almost complete   answer to the  open question on the correctness of $B_{k}^{m} 
\subsetneqq B_{k}$ for $k>1$,   posed in  \cite{startpaper}. We showed that  $B_{k}^{m} 
\subsetneqq B_{k}$ holds  for  all $k>1$ except for $k\in \{4,6\}$.
Of course, it is a pressing question to prove $B_{k}^{m} \subsetneqq B_{k}$  also for the remaining cases $k=4$ and $k=6$.
In order to prove $B_{4}^{m} \subsetneqq B_{4}$ by using  \autoref{mlbl} it would be enough to show that
$K_{4,49} \in B_{4}$ or  $K_{5,36} \in B_{4}$. In the case of  $k=6$ it would
suffice to show that $K_{5,102}$, $K_{6,71}$ or  $K_{7,63}$ is in $B_{6}$.
\smallskip

Additionally,  we considered the relationship of $B_k$ and $B_\ell^m$ for 
$\ell > k$. In this context the existence and the identification of
non-trivial functions $f, g\colon \nz\to \nz$
 such that $b^m(G)\leqslant f(b(G))$ and $b(G)\leqslant g(b^m(G))$ holds for 
 any graph
$G$ (or for any graph belonging to some particular class of graphs) is a general  question the
answer of which seems to be out of reach at the moment. 
However, we could deal with some  specific problems related to 
that question.  

In particular, we showed that for every $k \geqslant 5$ there is a graph in  $B_{k}$ which is 
not in $B_{2k-9}^{m}$, proving that $B_k \not \subseteq B_{2k-9}^m$ holds. 
In terms of the function $f$ above this implies $f(x)\geqslant 2x-8$ for  all $x\geqslant 5$, $x\in \nz$. 
Furthermore, we deduced that   $B_{k+1}^{m}\subseteq B_{k}$ does not hold in general by providing
 a graph that is $B_{2m-2}^m$ but not in $B_{2m-3}$ for every $m\geqslant 3$. 
 This implies that $g(2x) \geqslant 2x$ for   all $x\geqslant 2$,  $x\in \nz$ for the above function $g$. 
 
 Further, we showed that $B_{1}   \subseteq B_{3}^{m}$,  but we do not know whether this result is the best possible,
i.e.\  whether there is a graph in $B_{1}$ which is not in $B_{2}^{m}$ or whether  $B_{1} \subseteq B_{2}^{m}$ holds.

Another natural question which seems to be  simple but has  not been answered
 yet concerns the  inclusion    $B_{k}^{m} \subseteq B_{k+1}^{m}$. 
 We conjecture  this inclusion to be  strict, that is we conjecture that
 $B_{k}^{m} \subsetneqq B_{k+1}^{m}$ holds. A possible approach to prove this
 conjecture for a given $k\in \nz$ would be to  specify a particular pair of  natural numbers $(m,n)$
 with 
 $3 \leqslant m \leqslant n$ for which (a)   some Lower-Bound-Lemma implies 
 $K_{m,n}\not\in
 B_k^m$  and (b) a $B_{k+1}^m$-EPG representation can be constructed.   
The identification of such a pair  $(m,n)$,  $3 \leqslant m \leqslant n$, would 
clearly  prove  the existence of a
 complete bipartite graph $K_{m,n}$  with monotonic bend number equal to $k$ for any
 $k\geqslant 2$.

Finally,  the size of (monotonic)
EPG representations is another subject of interest.  In particular,   it would 
be
interesting to determine the minimum number of  grid lines 
  needed for a  $B_{k}$-EPG representation and $B_{\ell}^{m}$-EPG
representation of a graph $G$ with $b(G)\leqslant k$ and $b^m(G)\leqslant 
\ell$, respectively.


\bibliographystyle{plain}
\bibliography{papers}

\end{document}